\documentclass[12pt]{article}
\usepackage[lofdepth,lotdepth,caption=false]{subfig}
\usepackage{fancyhdr}
\usepackage{hyperref}
\usepackage{amsmath, amssymb, graphicx,amsthm}
\usepackage{xspace}
\usepackage{braket}
\usepackage{color}
\usepackage{setspace}
\usepackage{amscd}
\usepackage{amsxtra}
\usepackage[mathscr]{eucal}
\usepackage{mathrsfs}
\usepackage{upgreek}
\usepackage{indentfirst}

\def\R{\mbb{R}}
\def\C{\mbb{C}}

\newcommand{\be}{\begin{equation}}
\newcommand{\ee}{\end{equation}}
\newcommand{\mbb}[1]{\mathbb{#1}}

\numberwithin{equation}{section}

\usepackage{thmtools}
\usepackage{thm-restate}

\usepackage{hyperref}

\usepackage{cleveref}

\newtheorem{thm}{Theorem}[section]
\newtheorem{defi}[thm]{Definition}
\newtheorem{prop}[thm]{Propostion}
\newtheorem{cor}[thm]{Corollary}
\newtheorem{rmk}[thm]{Remark}
\newtheorem{lem}[thm]{Lemma}

\newcommand{\abs}[1]{\left\lvert#1\right\rvert}
\newcommand{\norm}[1]{\lVert#1\rVert}
\newcommand{\inpro}[1]{\langle#1\rangle}
\newcommand{\inppro}[1]{\left(#1\right)}
\newcommand{\sfl}[1]{\mathrm{sf}\{#1\}}
\newcommand{\Tr}[1]{\mathrm{Tr}\left[ #1 \right]}
\newcommand{\Trz}[1]{\mathrm{Tr}_z\left[ #1 \right]}
\newcommand{\tr}[1]{\mathrm{tr}\left[#1\right]}
\newcommand{\trz}[1]{\mathrm{tr}_z\left[#1\right]}
\newcommand{\gs}[1]{\delta_{#1}}
\newcommand{\g}[2]{\delta_{#1} (#2) \delta_{#1}^{-1}}
\newcommand{\supp}[1]{\mathrm{supp}(#1)}

\newcommand{\ed}[1]{\mathrm{End}(#1)}
\newcommand{\Ker}[1]{\mathrm{Ker}\{#1\}}

\def\A{\mathscr{A}}
\def\D{\mathscr{D}}
\def\M{\mathscr{M}}
\def\K{\mathscr{K}}
\def\R{\mathscr{R}}
\def\F{\mathscr{F}}
\def\P{\mathscr{P}}
\def\L{\mathscr{L}}
\def\d{\mathrm{d}}
\def\End{\mathrm{End}}

\begin{document}
\title{Asymptotic Spectral Flow}
\date{}
\author{Xianzhe Dai\footnote{
	Department of Mathematics, 
	University of California, Santa Barbara
	CA 93106,
	USA.  dai@math.ucsb.edu  \ \  Partially supported by the Simons Foundation}
 \ \ \ and \ \ \ Yihan Li\footnote{
	Chern Institute of Mathematics \& LPMC, 
	Nankai University,
	Tianjin 300071,
	China.  yhli@nankai.edu.cn} 
}
\maketitle

\textbf{Abstract}: In this paper we study the asymptotic behavior of the spectral flow of a one-parameter family $\{D_s\}$ of Dirac operators acting on the spinor bunldle $S$ twisted by a vector bundle $E$ of rank $k$, with the parameter $s\in [0,r]$ when $r$ gets sufficiently large. Our method uses the variation of eta invariant and local index theory technique. The key is a uniform estimate of the eta invariant $\bar{\eta}(D_r)$ which is established via local index theory technique and heat kernel estimate.

\section{Introduction}

The spectral flow of a one-parameter family of Dirac operators is first introduced by Atiyah-Patodi-Singer in their study of index theorems for Dirac operators over manifolds with boundary\cite{APS1} \cite{APS2} \cite{APS3}. It is closely related to the $\eta$-invariant, which is defined for a Dirac operator in the same theorem as the boundary correction term. Both spectral flow and $\eta$-invariant have  found significant application in diverse fields in mathematics and physics.

In this paper, we study the asymptotic behavior of spectral flow and $\eta$-invariant. More precisely, consider a closed spin manifold $M$ of an odd dimension $n$, equipped with its spinor bundle $S$ and a Hermitian vector bundle $E$ of rank $k$. Let $\nabla^E$ be a unitary connection on $E$ and $a\in \Omega^1(M,\End(E))$, a $\End(E)$-valued one-form on $M$, so that $\{\nabla^E+s a\}_{s\in [0,r]}$ is a one-parameter family of unitary connections on $E$, and, therefore, $\{\bar{\nabla}_s=\nabla^S \otimes 1+1 \otimes (\nabla^E + s {a})\}_{s\in[0,r]}$ forms a one-parameter family of unitary connctions on the bundle $S \otimes E$, where $\nabla^S$ denotes the connection on $S$ induced by the Levi-Civita connection on the tangent bundle $TM$. Then this family of connetions induces a one-parameter family $\{D_s\}$ of Dirac operators on $S \otimes E$. For simplicity we write the family of connection as $\{\nabla_s=\nabla_0 + s \hat{a}\}_{s\in[0,1]}$ (note the rescaled $s$) by setting $\hat{a}=ra$, and accordingly $D_s=D_0+sc(\hat{a})$,  and denote by $F_s$ the curvature of $\nabla_s$. The main problem discussed in this paper is the asymptotic behavior of the spectral flow of this family. This problem is initiated by Taubes in his proof of Weinstein conjecture in dimension 3 \cite{T2} and is later discussed for the general cases in \cite{T1}. Our main result, stated below,  provides improvement of the estimate in \cite{T1} for the general cases.

\begin{restatable}{thmm}{maina}
\label{main1}
Let $M$ be an odd dimensional closed spin manifold, and $D_0$ be a Dirac operator on it, and $D_s=D_0+sc(\hat{a})$, $0 \leqslant s \leqslant 1$, be the smooth curve of Dirac operators, where $\hat{a}=ra$ is a locally Lie algebra $\mathfrak{u}(k)$-valued 1-form on M with parameter $r\geqslant 1$. Set $R=\sup_M\{\abs{F_1}\}$, then there exists a constant $C>0$, such that the spectral flow satisfies
\[
\abs{\sfl{D_s,[0,1]}-(\frac{1}{2\pi \sqrt{-1}})^{\frac{n+1}{2}}\int_0^1 \int_M  \hat{A}(M) \wedge \tr{ \hat{a}\wedge \exp{(F_s)}} \d s}\leqslant CR^{\frac{n}{2}},
\]
when $r$ is sufficiently large.
\end{restatable}
Generally speaking, for a one-parameter family $\{D_s\}_{s\in [0,1]}$ of Dirac operators, spectral flow is the net number of eigenvalues that change sign while the parameter $s$ varies from 0 to 1. On the other hand, the $\eta$-invariant $\eta(D_s)$ for each single Dirac operator is a measure of its spectral asymmetry. From this point of view, given a path of Dirac operators, spectral flow is naturally related to the change of $\eta$-invariant. To be more precise, the relation between spectral flow and the change of (reduced) $\eta$-invariant over the interval $[0,1]$, is given by the formula:

\begin{equation}\label{esf}
\sfl{D_s}=-\int_0^1 \frac{\d\bar{\eta}(D_s)}{\d s}\d s+\bar{\eta}(D_1)-\bar{\eta}(D_0),
\end{equation}
where $\bar{\eta}(D_s)=\frac{1}{2}(\eta(D_s)+\dim \Ker{D_s})$ is called the reduced $\eta$-invariant in the sense of Atiyah-Patodi-Singer \cite{APS1}. Based on this relation, the proof of Theorem \ref{main1} consists of the explicit calculation of the variation $\frac{\d}{\d s} \bar{\eta}(D_s)$ and the estimate of $\bar{\eta}(D_1)$.

In Section 3, there provides an explicit computation for the variation of $\eta$-invatiant, which appears in the first term on the right-hand side of \eqref{esf}. A fundamental observation is that the relation (\ref{esf}) can also be viewed as a special case of the Atiyah-Patodi-Singer index theorem by applying it on $M\times [0,1]$. From this viewpoint, the variation of $\eta$-invariant is calculated explicitly using the technique of local index theorem, Getzler rescaling introduced in \cite{G} \cite{getzler1993odd}. Getzler's rescaling technique also plays a crucial role in the later estimate of $\eta$-invariant.

The second part of the estimate, done in Section 4, mainly focuses on $\bar{\eta}(D_1)$ and especially its dependence on the parameter $r$, which is considered to be large enough. Based on the formula
\[
\eta(D_1)=\frac{1}{\sqrt{\pi}} \int_0^{\infty} t^{-\frac{1}{2}} \Tr{D_1 \exp{(-tD_1^2)}}\d t,
\]
the estimate is separated into two parts: the small-time part $\int_0^{t_0} t^{-\frac{1}{2}} \Tr{D_1 \exp{(-tD_1^2)}}\d t$, and the large-time part $\int_{t_0}^{\infty} t^{-\frac{1}{2}} \Tr{D_1 \exp{(-tD_1^2)}}\d t$, with the value of $t_0$ to be properly chosen. For the small-time part, the basic idea, following \cite{BL}, is applying Getzler's rescaling technique to approximate the trace $\Tr{D_1 \exp{(-tD_1^2)}}$ for small time $t$'s by some rescaled heat kernel at the fixed ``time". Inspired by \cite{BL} and \cite{DLM}, we developed a unifrom estimate for such a family of ``rescaled" heat kernel and therefore provide the estimate of the small-time part. Here we would like to point out a subtle difference from \cite{BL} in the setting, which complicates the problem. As mentioned above, the operator $D_1$ differs from $D_0$ by the Clifford action $c(\hat{a})$ of a $\End(E)$-valued 1-form $\hat{a}$, while in \cite{BL}, the operators of interest are of the type $D+V$, where $V$ is an $\End(E)$-valued function. And we introduce some new trick in Section 4 to extract the contribution by the parameter $r$ through the estimate.

On the other hand, the large-time part, together with $\mathrm{dim}(\mathrm{Ker}\{D_1\})$, can be controlled by the trace $\Tr{\exp{(-t_0 D_1^2)}}$, whose estimate could be provided by the heat kernel estimate in $\cite{T1}$. The main result of this part can be summarized as

\begin{restatable}{thmm}{mainb}
\label{main2}
Let $M$ be an odd dimensional closed spin manifold, and $D$ be a Dirac operator acting on the bundle $S\otimes E$,  $a$ be a locally Lie algebra $\mathfrak{u}(k)$-valued one-form on $M$, and $r\geqslant 1$. Set $R=\sup_M\{\abs{F_1}\}$. Then there exists a constant $C>0$, such that
\[
\abs{\bar{\eta}(D+rc(a)) }\leqslant CR^{\frac{n}{2}},
\]
when $r$ is sufficiently large.
\end{restatable}

This problem, the asymptotic of $\eta$-invariant, has also been discussed recently by Savale in \cite{NS} and \cite{NS2}, where he discussed in detail the asymptotic of the eta invariant of the Dirac operators $D_s$ acting on $S\otimes L$ where $L$ is a line bundle. In \cite{NS}, it is shown that $\bar{\eta}(D+rc(a))\sim o(r^{\frac{n}{2}})$. In 2018, he has improved the estimate of the eta invariant in \cite{NS2} to $\eta(D+rc(a)) \sim O(r^{\frac{n-1}{2}})$ under some extra assumptions. In the cases where the parameter $r$ contributes linearly in the curvature $F_1$ of the unitary connection $\nabla_1$, including the case in \cite{NS} and \cite{NS2}, the result of Theorem \ref{main2} implies that $\bar{\eta}(D+rc(a))\sim O(r^\frac{n}{2})$.

\section{Preliminaries}

This chapter is a short review of some well-known results on Clifford algebras and the spin geometry that is involved in the proof of the main result.

\subsection{Clifford algebra and the space of spinors}
\begin{defi}
Let $V$ be a vector space over a field $\mathbf{K}$ of dimension $n$ endowed with a non-degenerate bilinear form $g$. The Clifford algebra $Cl(V,g)$ associated to $g$ on $V$ is an associative algebra with unit, defined by
\[
Cl(V,g):=T(V)/I(V,g)
\]
where $T(V)=\bigoplus_{r=1}^\infty (\otimes^r V)$ is the tensor algebra of $V$, and $I(V,g)$ is the ideal generated by all the elements of the form $v\otimes v +g(v,v)1$.
\end{defi}

In particular, denote by $Cl(n)$ the Cliffod algebra of $\mathbb{R}^n$ associated to its canonical metric and denote by $S_n$ be the complex Hermitian space of spinors such that $\End(S_n) \simeq \mathbb{C}l(n)=Cl(n)\otimes_{\mathbb{R}}\mathbb{C}$.

Because of the main object in this paper, our discussion is restricted to the case $n$ is odd. A property of special importance through this paper is that the trace $\mathrm{tr}[\cdot]$ behaves on the odd elements of $\mathbb{C}l(n)$ in the exactly same way as the supertrace $\mathrm{tr}_s$ on the even elements of $\C l(n)$ for $n$ even, i.e. 
\begin{equation}
\tr{1}=2^{[\frac{n}{2}]}, \qquad \tr{c(e_1)\dots c(e_n)}=2^{[\frac{n}{2}]}(-\sqrt{-1})^{\frac{n+1}{2}}
\end{equation}
and that the trace of the other monomials in $\C l(n)$ is 0.

\subsection{A family of Dirac operators}
Note that the Clifford action $c:\mathbb{C}l(n) \rightarrow \End(S_n)$ determines uniquely a representation $\rho : Spin(n) \rightarrow \End(S_n)$. For the spin manifold $M$, let $P_{Spin(n)}$ be the principal $Spin(n)$-bundle. Define spinor bundle $S$ by
\[
S=P_{Spin(n)}\times_\rho S_n.
\]
And equip it with the unitary connection $\nabla^S$ uniquely determined by the Levi-Civita connection $\nabla^{TM}$ \cite[Thereom 4.14]{Spin}.

In addition, let $E$ be a Hermitian vector bundle of rank $k$ over $M$, and equip it with a unitary connection $\nabla^E_0$. Let $a\in \Omega^1(M,\End(E))$ be a $1$-form such that $\nabla^E_0+a$ is another unitary connection on $E$. The $1$-form $a$ is locally Lie algebra $\mathfrak{u}(k)$ valued in the sense that, in terms of any local trivialization of $E$, the $1$-form $a$ takes its value in $\mathfrak{u}(k)$. Then let $\hat{a}=r\cdot a$ for some real number $r \geqslant 1$, and set $\nabla^E_s= \nabla^E_0 + s\cdot \hat{a}$, which defines a one-parameter family $\{\nabla^E_s\}_{s\in [0,1]}$ of unitary connections on $E$. Now equip the vector bundle $S\otimes E$ with a family $\{\nabla_s\}_{s\in[0,1]}$ of unitary connections defined by
\[
\nabla_s=\nabla^S\otimes 1+1 \otimes \nabla^E_s.
\]
This induces a family $\{D_s\}_{s\in [0,1]}$ of (twisted) Dirac operators, such that, in terms of an arbitrary local orthonormal frame $\{e_i\}_{i=1}^n$ on the manifold $M$, $D_s$ can be expressed as
\[
D_s=\sum_{i=1}^n c(e_i)\nabla_{s, e_i}.
\]

\subsection{Spectral flow and $\eta$-invariant}
The concept of spectral flow for a smooth family of Dirac operators is first introduced by Atiyah-Patodi-Singer in their study of index theory on manifolds with boundary in \cite{APS1}\cite{APS3}.

\begin{defi}
If $D_s$, $0\leqslant s \leqslant 1$ is a curve of self-adjoint Fredholm operators, the spectral flow $\sfl{D_s}$ counts the net number of eigenvalues of $D_s$ which change sign when $s$ varies from 0 to 1.
\end{defi}

It was summarized in \cite{DZ} the following properties of spectral flow.

\begin{prop}\label{prptsf}
The spectral flow has the following properties:
\begin{itemize}
\item[(1)] If $D_s$, $0 \leqslant s \leqslant 1$, is a curve of self-adjoint Fredholm operators, and $\uptau \in [0,1]$, then

\begin{equation}
\sfl{D_s,\, [0,1]}=\sfl{D_s,\, [0,\uptau]}+\sfl{D_s,\, [\uptau,1]}.
\end{equation}

\item[(2)] If $D_s$, $0\leqslant u \leqslant 1$, is a smooth curve of self-adjoint elliptic pseudodifferential operators on a closed manifold, and  $\bar{\eta}(D_s)=\frac12 (\eta(D_s)+dim\,Ker D_s)$ is the reduced $\eta$ invariant of $D_s$ in the sense of Atiyah-Patodi-Singer, then $\eta$ is smooth mode $\mathbb{Z}$ and

\begin{equation}\label{sf}
\sfl{D_s}=-\int_0^1 \frac{\d\bar{\eta}(D_s)}{\d s}\d s+\bar{\eta}(D_1)-\bar{\eta}(D_0).
\end{equation}

\item[(3)] If $D_s$, $0\leqslant s \leqslant 1$, is a periodic one-parameter family of self adjoint Dirac-type operators on a closed manifold, and $\tilde{D}$ is the corresponding Dirac-type operator on the mapping torus, then

\begin{equation}
\sfl{D_s}=\mathrm{ind} \tilde{D}
\end{equation}
\end{itemize}
\end{prop}

\subsection{Heat kernel and $\eta$-invariant}
It can be seen from the previous section that spectral flow is closely related to $\eta$-invariant. And, indeed, throughout this paper, $\eta$-invariant plays a fundamental rule. So in this section, we briefly revisit the concept of $\eta$-invariant, and in particular its relation with heat kernel. 

First of all, for a single Dirac operator $D_s$ ($s\in [0,1]$), we define the corresponding $\eta$-function of a complex variable $z$ by
\[
\begin{split}
\eta_{D_s}(z) &=\sum_{\lambda \in Spec\{D_s\}} sgn(\lambda) \abs{\lambda}^{-z}=\sum_{\lambda \in Spec\{D_s\}} \frac{\lambda}{\abs{\lambda}}\abs{\lambda}^{-z}\\
&=\sum_{\lambda \in Spec\{D_s\}} \lambda \abs{\lambda}^{-z-1},
\end{split}
\]
which can be written as
\[
\eta_{D_s}(z)=\sum_{\lambda \in Spec\{D_s\}} \lambda (\lambda^2)^{-\frac{z+1}{2}}.
\]

It can be shown that this function has a meromorphic extension on the complex plane for $z\in \mathbb{C}$, and, in particular, it is analytic at $z=0$ which allows us to define the $\eta$-invariant by evaluating the $\eta$-function at $z=0$ as

\[
\eta(D_s)=\eta_{D_s}(0).
\]

Noticing that for any real number $\lambda\neq 0$, one has
\[
\lambda^{-z}=\frac{1}{\Gamma(z)}\int_0^\infty t^{z-1}e^{-t\lambda}\d t,
\]
it can be written that
\begin{align*}
\eta_{D_s}(z) 	&=\sum_{\lambda_s\in Spec(D_s)}\lambda_s \frac{1}{\Gamma(\frac{z+1}{2})}\int_0^\infty t^{\frac{z-1}{2}}e^{-t\lambda_s^2}\d t\\
			&=\frac{1}{\Gamma(\frac{z+1}{2})}\int_0^\infty t^{\frac{z-1}{2}}\sum_{\lambda_s\in Spec(D_s)}\lambda_s e^{-t\lambda_s^2}\d t\\
		     	&=\frac{1}{\Gamma(\frac{z+1}{2})}\int_0^\infty t^{\frac{z-1}{2}}\Tr{D_s \exp{(-tD_s^2)}}\d t.
\end{align*}

It follows immediately, by evaluating $\eta_{D_s}(z)$ the function at $z=0$, that the $\eta$-invariant $\eta(D_s)$ can be written as
\[
\eta(D_s)=\frac{1}{\sqrt{\pi}} \int_0^{\infty} t^{-\frac{1}{2}} \Tr{D_s \exp{(-tD_s^2)}}\d t.
\]

These equations naturally involves the heat operator $\exp{(-tD^2)}$. Our motivation to look at the heat kernel comes from the Lidskii's Theorem which identified the trace of an operator with the integral of the trace of its kernel.

\begin{prop}[Lidskii's Theorem]\label{lidskii}
If $T: L^2(S\otimes E) \rightarrow L^2(S\otimes E)$ is defined by a continuous kernel function $T(x; y)$ by $(Tf)(x)=\int_M T(x,y)f(y)dy$,
then

\[
\Tr{T}= \int_M \tr{T(x,x)} \d x.
\]
\end{prop}
For the heat operator $\exp{(-tD_s^2)}$, the kernel $K_{s}(t;x,y)$ of it is defined by,
\begin{equation}
(\exp{(-tD_s^2)}\sigma)(x)=\int_M K_{s}(t;x,y)\sigma(y)\d{y} \qquad \forall \sigma \in \Gamma (S\otimes E).
\end{equation}
In the following sections, it will mainly involves the estimate of $\Tr{D_s \exp{(-tD_s^2)}}$ and $\Tr{D_1 \exp{(-tD_1)}}$, which will be done by dealing with their corresponding kernels.

\section{Variation of $\eta$-invariant}

It is a well-known phenomenon of the $\eta$-invariant that although the $\eta$-invariant itself is not locally computable, its variation is. As mentioned above, the variation $ \dot{\bar{\eta}}(D_s)$ in \eqref{sf} can be viewed as the integral term in the APS-index theorem \cite{APS1}, where the variation of $\bar{\eta}(D_s)$ corresponds to the integrand. This observation motivates us to apply the technique for the local index theorem introduced by Getzler in \cite{G} to provide an explicit computation which is done as follows.

\subsection{Variation of $\eta$-invariant}
Here we follow the procedure as \cite{getzler1993odd} to express the variation formula for $\bar{\eta}(D_s)$ in terms of heat operator.
\begin{prop}
The reduced $\eta$-invariant is a smooth function in the parameter $s\in [0,1]$ in the sense of modulo $\mathbb{Z}$, and its variation is given by the following formula.
\begin{equation}\label{varfml}
\frac{\d}{\d s}\bar{\eta}(D_s)= \lim_{t\rightarrow 0}t^{\frac{1}{2}} \Tr{\dot{D}_s \exp{(-tD_s^2)}}.
\end{equation}
\end{prop}
\begin{proof}
Without loss of generality, we assume that there are only finitely many value of $s$ such that $D_s$ in uninvertible. 

Here we show at first the formula for the case that $D_s$ is invertible. Under this assumption, one has $\bar{\eta}(D_s)=\frac{1}{2}\eta(D_s)$, and therefore
\begin{equation}\label{vardrvt}
\begin{split}
\frac{\d}{\d s}\bar{\eta}(D_s)=&\frac{1}{2}\frac{\d}{\d s}\int_0^{\infty} t^{-\frac{1}{2}} \Tr{D_s \exp{(-t D_s^2)}}\d t\\
=&\frac{1}{2}\int_0^{\infty} t^{-\frac{1}{2}} \Tr{\dot{D}_s \exp{(-t D_s^2})}\d t+\int_0^\infty t^{\frac{1}{2}}\Tr{\dot{D}_s D_s^2 \exp{(-t D_s^2)}}\d t.
\end{split}
\end{equation}
On the other hand, from integration by parts, it follows that
\begin{equation}\label{varibp}
\begin{split}
&\int_0^\infty t^{\frac{1}{2}}\Tr{\dot{D}_s D_s^2 \exp{(-t D_s^2)}}\d t=\int_0^\infty t^{\frac{1}{2}}\frac{\partial}{\partial t}\Tr{\dot{D}_s \exp{(-t D_s^2)}}\d t\\
=& \left.t^{\frac{1}{2}}\Tr{\dot{D}_s \exp{(-t D_s^2)}}\right\vert_0^\infty-\frac{1}{2}\int_0^{\infty} t^{-\frac{1}{2}} \Tr{\dot{D}_s \exp{(-t D_s^2})}\d t.
\end{split}
\end{equation}
Futhermore, by the assumption that $D_s$ is invertible,
\begin{equation}\label{varliminft}
\lim_{t\rightarrow \infty} t^{\frac{1}{2}}\Tr{\dot{D}_s \exp{(-t D_s^2)}}=0.
\end{equation}
Therefore, by substituting \eqref{varibp} and \eqref{varliminft} into \eqref{vardrvt}, one has
\begin{equation}\label{varinvtl}
\frac{\d}{\d s}\bar{\eta}(D_s)=\lim_{t\rightarrow 0} t^{\frac{1}{2}}\Tr{\dot{D}_s \exp{(-t D_s^2)}}.
\end{equation}

It now remains to extend \eqref{varinvtl} to the values of $s$ where $D_s$ is uninvertible. Assuming that $D_{\bar{s}}$ is uninvertible for some $\bar{s} \in (0,1)$, then there exists an $\varepsilon>0$ such that $0$ is the only eigenvalue of $D_{\bar{s}}$ within the interval $[-\varepsilon,\varepsilon]$. Furthermore, according to \cite[\S 17A]{booss1993elliptic}, there exists a $\delta>0$ such that, for all $s\in (\bar{s}-\delta,\bar{s}+\varepsilon)\setminus \{\bar{s}\}$, $D_s$ is invertible and neither of $\pm\varepsilon$ is an eigenvalue of any of such $D_s$'s. This provides a decomposition $L^2(M, S\otimes E)=\Ker{D_{\bar{s}}}\oplus\Ker{D_{\bar{s}}}^{\perp}$, under which the operators $D_s$ is block diagonal whenever $s\in (\bar{s}-\delta,\bar{s}+\varepsilon)$.

Denoting by $D_s^{\perp}$ the restriction of the operator $D_s$ on $\Ker{D_{\bar{s}}}^{\perp}$, the reduced $\eta$-invariant $\bar{\eta}(D_s)$ can then be written as
\begin{equation}
\bar{\eta}(D_s)=\bar{\eta}(D_s^{\perp})+\sum_{\abs{\lambda}<\varepsilon} \mathrm{sign}(\lambda)\mathrm{dim}E(D_s,\lambda),
\end{equation}
where $E(D_s,\lambda)$ is the eigenspace of the operator $D_s$ associated to the eigenvalue $\lambda$. Then $\sum_{\abs{\lambda}<\varepsilon} \mathrm{sign}(\lambda)\mathrm{dim}E(D_s,\lambda)$ is a $\mathbb{Z}$-valued step function, whose value changes only at $\bar{s}$. While, by applying the previous argument, the function $\bar{\eta}(D_x^{\perp})$ is smooth in the parameter $s$. And for all $s\in (\bar{s}-\delta,\bar{s}+\varepsilon)\setminus \{\bar{s}\}$, the derivative of $\bar{\eta}(D_x^{\perp})$ coincides with that of $\bar{\eta}(D_s)$. This allow us to extend the formula \eqref{varinvtl} as desired.
\end{proof}

\subsection{Getzler's rescaling}
Let $\delta>0$ be less than the injectivity radius of the manifold $(M,g)$. Given an $x_0\in M$, denote by $B^M(x_0,\delta)$ and $U=B^{T_{x_0} M}(0,\delta)$ the open balls in $M$ and $T_{x_0}M$ centered at $x_0$ with radius $\delta$ respectively. Then the exponential map $\exp_{x_0}: T_{x_0}M \rightarrow M$ identifies $U$ diffeomorphically with $B^M(x_0, \delta)$.  Furthermore, for $x=\exp_{x_0} (X)$ we identify $S_x, \, E_x$ to $S_{x_0}$ and $E_{x_0}$ by parallel transport with respect to the connections $\nabla^S$ and $\nabla_s^E$ along the curve $\gamma: [0, 1] \ni u \rightarrow \exp_{x_0}(uX)$. 

Choose an oriented orthonormal basis $\{e_i\}_{i=1}^n$ of $T_{x_0}M$, and denote by $\{e^i\}_{i=1}^n$ its dual basis. Given $X \in U$, let $\tilde{e}_i(X)$ be the parallel transport of $e_i$ with respect to $\nabla^{TM}$ along the geodesic $\gamma$. Then note that $[\nabla_{X}, c(\tilde{e}_i)]=[c(\nabla^{TM}_X\tilde{e}_i)]=0$, for all $X\in U$. Hence the Clifford action $c(\tilde{e}_i)(X)$ is identified with $c(e_i)\in \ed{S_{x_0}}$ under such trivialization. We also identify $S_{x_0}$ with the exterior algebra $\Lambda ^\ast (T_{x_0}^\ast M)$ with the Clifford action defined via
\begin{equation}\label{clfdact}
c(e_i)=\varepsilon (e^i)-\iota(e_i),
\end{equation}
where $\varepsilon$ is the exterior product map and $\iota$ is the contraction. 

Now denote the induced identification for $X\in U$ by the map $\tau_s(x_0,X):S_x\otimes E_x \rightarrow {S_{x_0}}\otimes {E_{x_0}}$ for $x=\exp_{x_0}(X)$. And denote by $K_s(t;x,y)$ the kernel of the heat operator $\exp{(-tD_s^2)}$ and set
\begin{equation}\label{htkn}
\K_s(t;X)=\tau_s(x_0,x) K_s(t; x,x_0)  \qquad \textrm{for } x=\exp_{x_0}(X),
\end{equation}
which defines a map from $\mathbb{R}_+\times U$ to $\ed{S_{x_0}}\otimes \ed{E_{x_0}}$. Therefore, it can be written as
\begin{equation}
\K_s(t;X)=\sum_{\abs{I} \textrm{ even}} a_{s,I}(t;X) c(e_I),
\end{equation}
where each $I=(i_1,i_2,\dots, i_k)$ is a multi-index and $a_{s,I}(t,X)\in \ed{E_{x_0}}$.

Hence, for the operator $\dot{D}_s\exp{(-tD_s^2)}$, whose kernel is exactly $c(\hat{a})K_{s}(t;x,y)$, by \eqref{htkn}, consider
\begin{equation}
\begin{split}
t^{\frac 12}\sqrt{-1}c(\hat{a})\K_{s}(t;X)&=c(\hat{a})\sum_{\abs{I}\text{ even}} a_{s,I}(t;X) c(e_I)\\
&=\sum_{\abs{I} \text{ odd}} b_{s,I}(t;X)c(e_I).
\end{split}
\end{equation}

In the Clifford algebra $\mathbb{C}l(n)$ with $n$ odd, recall a property of essential importance that
\begin{equation}
\tr{1} = 2^{\frac{n-1}{2}} ,\qquad \tr{c(e_1)\dots c(e_n)}= 2^{\frac{n-1}{2}}(-\sqrt{-1})^{\frac{n+1}{2}},
\end{equation}
while the trace of the other monomials in $\C l(n)$ is 0. It then follows that only the top degree term contribute to the trace of the kernel, i.e.
\begin{equation}\label{trclfd}
\begin{split}
&\tr{c(\hat{a})\K_{s}(t;X)}=\tr{b_{I_n}(t;X)}\tr{c(e_{I_n})}\\
=&2^{\frac{n-1}{2}}(-\sqrt{-1})^{\frac{n+1}{2}}\mathrm{tr}[b_{I_n}(t;X)],
\end{split}
\end{equation}
where by $I_n$ we denote the multi-index $(1,2,\dots,n)$.

The Getzler rescaling possesses a similar property, capturing the McKean-Singer ``fantastic cancellation". In addition to its action on function on the coordinate components for a function $f$ on $\mathbb{R}_{+}\times \mathbb{R}^n $ as $(\delta_t f)(u,x)=f( t u, t^{\frac 12}x)$, it also applies a rescaling on the action of Clifford algebra, which incorporates the Clifford degrees. As mentioned above, we can identify $S_{x_0}$ with $\Lambda^\ast (T_{x_0}^\ast M)$ and the effect of Getzler rescaling on it be determined by
\begin{equation}
\delta_t(e^{i_1}\wedge e^{i_2}\wedge \dots \wedge e^{i_p})=t^{-\frac{p}{2}}e^{i_1}\wedge e^{i_2}\wedge \dots \wedge e^{i_p}.
\end{equation}
This induces the effect on Clifford action in the following way
\begin{equation}
c_t(e_i)=\g{t}{c(e_i)}=t^{-\frac{1}{2}}\varepsilon(e)^i-t^{\frac{1}{2}}\iota(e_i).
\end{equation}

Taking limit of the a rescaled Clifford action, it follows that
\begin{equation}\label{limrscl}
\lim_{\varepsilon \rightarrow 0} t^{\frac n2}\g{t}{c(e_I)}=
\begin{cases}
0, & \text{if } \abs{I}<n;\\
\varepsilon(e^1\wedge e^2 \wedge \dots \wedge e^n), & \text{if } I=(1,2,\dots, n).
\end{cases}
\end{equation}

Combining \eqref{trclfd} and\eqref{limrscl}, the problem is now to compute the limit of  $t^{\frac{n}{2}}(\delta_t \K_s)(1,X)$ as $t\rightarrow 0$ by the following statement.

\begin{lem}
The trace of the kernel function $c(\hat{a})K_{s}(t;x_0,x_0)$ can be calculated by the following equation
\begin{equation}\label{trsclg}
\lim_{t\rightarrow 0}t^{\frac{1}{2}}\tr{c(\hat{a})K_{s}(t;x_0,x_0)}=2^{\frac{n-1}{2}}(-\sqrt{-1})^{\frac{n+1}{2}}\tr{\lim_{t \rightarrow 0}t^{\frac{n+1}{2}} c_{t}(\hat{a})(\delta_{t}\K_{s})(1;X)}.
\end{equation}
\end{lem}

\subsection{Computing the limit}
It now remains to evaluate the limit on the right-hand side of \eqref{trsclg}. Note that, by its construction, 
$t^{\frac{n+1}{2}} c_t(\hat{a})(\delta_t \K_{s})(1;X)$ is the kernel of operator $t^{\frac{1}{2}}c_t(\hat{a}) \exp{(-D_{s,t}^2)}$, where $D_{s,t}=t^{\frac{1}{2}}\g{t}{D_s}$. As in the proof of local index theorem (c.f. \cite{BGV}), the limiting operator $L_0=\lim_{t\rightarrow 0}D_{s,t}^2$ can be computed explicitly via Lichnerowicz formula. Furthermore, it is shown that the limit $\lim_{t\rightarrow 0}t^{\frac{n+1}{2}} c_t(\hat{a})(\delta_t \K_{s})(1;X)$ equals to the kernel of the operator $\varepsilon(\hat{a}) \exp{(-L_0)}$, which has a explicit expression by Mehler's formula.

We start by computing the limiting operator.
\begin{prop}\label{limrsd}

The limit of the rescaled operator $t \g{t}{D_s^2}$ as $t\rightarrow 0$ is:
\begin{equation}\label{limrsdf}
\lim_{t \rightarrow 0} t \g{t}{D_s^2}=\mathscr{L}+\varepsilon(F_s),
\end{equation}
where the operator
\[
\mathscr{L}=-(\partial_{i}+\frac14 \sum_{j=1}^n\Omega_{ij} X^j)^2
\]
is the generalized harmonic oscillator, and the $2$-form $\Omega_{ij}=\sum_{k<l}R_{ijkl}\varepsilon(e^k)\varepsilon(e^l)$.
\end{prop}

\begin{proof}
Recall that the Lichnerowicz formula for the operator $D_s$ says
\begin{equation}\label{lichnrw}
D_s^2=\nabla_s^\ast \nabla_s +c(F_s)+\frac{K}{4}.
\end{equation}
Applying Getzler rescaling on \eqref{lichnrw}, it follows that
\[
t \g{t}{D_s^2} =t\g{t}{\nabla_s^\ast \nabla_s}+\frac{t K_t}{4}+t c_t (F_s).
\]

It is straightforward, for the last two terms on the right-hand side, that
\[
\begin{split}
&\lim_{t \rightarrow 0} tK_t=0,\\
&\lim_{t \rightarrow 0} t c_{t} (F_s)=\varepsilon(F_s)=\sum_{i<j} F_s(\tilde{e}_i, \tilde{e}_j)\varepsilon(e^i)\varepsilon(e^j).
\end{split}
\]

It now remains to calculate $t\g{t}{\nabla_s^\ast \nabla_s}$, for which, in terms of the frame $\{\tilde{e}_i\}$, one has the expression
\begin{equation}
\nabla_s^\ast \nabla_s=\sum_{i=1}^n(-\nabla_{s,\tilde{e}_i}\nabla_{s,\tilde{e}_i}+\nabla_{s,\nabla^{TM}_{\tilde{e}_i}\tilde{e}_i}).
\end{equation}
   
On the other hand, in terms of the trivialization taken for the bundle $S\otimes E $ around $x_0\in M$, write the connection $\nabla_s$ as
\[
\nabla_s=\d+\omega+A_s,
\]
where $\omega$ and $A_s$ are connection one-forms of $S$ and $E$ respectively, and, accordingly,
\[
t^{\frac{1}{2}}\g{t}{\nabla_s}=\d+\omega_t+A_{s,t}.
\]
Writing $A_s=\sum_{i=1}^n \alpha_{s,i} \d X^i$, where each $\alpha_i$ is a $\mathfrak{u}(k)$-valued function that is determined by the curvature $F_s$ via
\[
\alpha_{s,i} (X)=\int_0^1 \rho X^j F_s(\partial_j,\partial_i)(\rho X)\d\rho.
\]
it then follows that
\begin{equation}\label{cnctform1}
\begin{split}
\alpha_{s,t,i}(X)&=t^{\frac{1}{2}}\alpha_i(t^{\frac{1}{2}}X)\\
&=t\int_0^1 \rho X^j F_s(\partial_j,\partial_i)(\rho t^{\frac{1}{2}}X)\d\rho,
\end{split}
\end{equation}
which implies
\begin{equation}
\lim_{t\rightarrow 0} A_{s,t}(X)=0.
\end{equation}
And, similarly,
\begin{equation}
\omega_t(\partial_i)\vert_X=t\int_0^1 \rho X^j R_t^{S}(\partial_j,\partial_i)(\rho t^{\frac{1}{2}}X)\d\rho,
\end{equation}
where $R^S$ is the curvature of the spinor bundle and 
\[
R_t^S(\partial_j,\partial_i)=\frac{1}{2}\sum_{k<l}\inpro{R^{TM}(\partial_j, \partial_i)\tilde{e}_k, \tilde{e}_l}c_t(e_k)c_t(e_l).
\]
Thus, setting $R_{ijkl}=\inpro{R^{TM}(e_i,e_j)e_k,e_l}$, then
\begin{equation}
\lim_{t\rightarrow 0}\omega_t(\partial_i)\vert_X=\frac{1}{4} \sum_{j;k<l}X^j R_{ijkl}\varepsilon(e^k) \varepsilon(e^l).
\end{equation}

It then follows that
\begin{equation}
\lim_{t\rightarrow 0}t \g{t}{\nabla_{s,\nabla^{TM}_{\tilde{e}_i}\tilde{e}_i}}=\lim_{t\rightarrow 0}t\sum_{k=1}^n\Gamma_{ii}^k\vert_X \nabla_{s,\tilde{e}_k}=0,
\end{equation}
where $\Gamma_{ij}^k=\inpro{\nabla^{TM}_{\tilde{e}_i}\tilde{e_j}, \tilde{e}_k}$ vanishes at $x_0$, and one finally has
\begin{equation}\label{limlaplace}
\begin{split}
\lim_{t \rightarrow 0}t \g{t}{\nabla_s^\ast \nabla_s} &=\lim_{t \rightarrow 0}t \g{t}{\nabla_{s,\tilde{e}_i}\nabla_{s,e_i}-\nabla_{s\nabla_{s,e_i}e_i}}\\
&=-(\partial_{i}+\frac{1}{4}\sum_{j}X^j\Omega_{ij})^2.
\end{split}
\end{equation}
\end{proof}

Recall the Mehler's formula which provides an explicit expression for the kernel of $\exp{(-u\L)}$.
\begin{lem}
Letting $\Omega=(\Omega_{ij})$ be the $n \times n$ matrix of such 2-forms, then for the generalized harmonic oscillator $\L$,  the kernel of $\exp{(-u\L)}$ is given by
\begin{equation}
K_{\L}(u;X,0)=\frac{1}{(4\pi u)^{\frac n2}}\mathrm{det}^{\frac{1}{2}}(\frac{u\Omega/2}{\sinh{\Omega/2}})\exp(-\frac{1}{4u}(\frac{u\Omega/2}{\tanh{u\Omega/2}})_{ij}X^i X^j).
\end{equation}
\end{lem}

Then for the limiting operator $\L+\varepsilon(F_s)$, one has the following result.
\begin{cor}
Thus the kernel of the operator $\exp{(-\L-\varepsilon(F_s))}$ is
\begin{equation}\label{knlhrmosl}
\K_{s,0}(1;X,0)=K_\mathscr{L}(1;X,0) \exp{(-F_s)}.
\end{equation}
\end{cor}

As in \cite{BGV}, the kernel given in \eqref{knlhrmosl} is related with the aim of this section by the following lemma.
\begin{lem}\label{limofker}
The limit of the rescaled kernel $t^{\frac{n+1}{2}}c_t{\hat{a}}(\delta_t \K_s)(1;X)$ is given by
\begin{equation}
\lim_{t\rightarrow 0} t^{\frac{n+1}{2}}c_t(\hat{a})(\delta_t \K_s)(1;X)=\varepsilon(\hat{a})\K_{s,0}(1;X,0).
\end{equation}
\end{lem}

Since the estimate of rescaled heat kernel will be systematically treated in the next section, we postpone the proof of this lemma to there. Here, we use the result to accomplish the variation formula for $\eta$-invariant.

\begin{prop}
The variation of $\eta$-invariant is
\begin{equation}
\frac{\d}{\d s}\bar{\eta}(D_s)=(\frac{1}{2\pi \sqrt{-1}})^{\frac{n+1}{2}}\int_M \hat{a}\wedge \hat{A}(\frac{\Omega}{2\pi})\wedge \tr{\exp(F_s)}.
\end{equation}
\end{prop}

\begin{proof}
Firstly, it follows from the lemma above that
\begin{align*}
\lim_{t \rightarrow 0} t^{\frac12}c_t(\hat{a})(\delta_t\K_{s})(1;X) &= (-2\sqrt{-1})^{\frac{n-1}2}\left[\lim_{t \rightarrow 0} t^{\frac {n+1}{2}} (\delta_t (c(\hat{a})\K_{s}))(1;X)\right]\\
										&= (\sqrt{-1})^{-1}(-2\sqrt{-1})^{\frac{n-1}2}[\lim_{t \rightarrow 0 }t^{\frac12} \hat{a}\wedge \K_{s,0}(t,0)]^{\rm{max}}\\
										&=\sqrt{\pi}(\frac{1}{2\pi \sqrt{-1}})^{\frac{n+1}{2}}[\hat{a}\wedge \hat{A}(\frac{\Omega}{2\pi})\wedge \exp(F_s)]^{\rm{max}}.
\end{align*}

On the other hand, from \eqref{varfml}, it follows that 

\begin{equation}
\begin{split}
\frac{\d}{\d s}\bar{\eta}(D_s)&=\frac{1}{\sqrt{\pi}}\int_M \tr{\lim_{t \rightarrow 0} t^{\frac12}c(\hat{a})K_{s}(t;0,0)}\\
&=(\frac{1}{2\pi \sqrt{-1}})^{\frac{n+1}{2}} \int_M\hat{A}(\frac{\Omega}{2\pi})\wedge \tr{\hat{a}\wedge \exp(F_s)}.
\end{split}
\end{equation}
\end{proof}

When the twisted part $E$ is a line bundle, i.e. $k=1$, $F_s=sr\d a$. The following estimate has been pointed out in \cite{T1}.
\begin{rmk}
When the twisted part $E$ is a line bundle, i.e. $k=1$, the 1-form $\hat{a}$ is purely imaginary-valued and it has been asserted in \cite{T1} that the leading order term in this part is given by
\begin{equation}
(\frac{1}{2\pi \sqrt{-1}})^{\frac{n+1}{2}} \frac{1}{(\frac{n+1}{2})!} r^{\frac{n+1}{2}}\int_M a \wedge (da)^{\frac{n-1}{2}},
\end{equation}
which also gives the leading order term of the asymptotic spectral flow.
\end{rmk}

When $k>1$, the $1$-form $\hat{a}$ is locally $\mathfrak{u}(k)$-valued. First write locally that $\nabla_0$ as $\nabla_0=\d+\omega$, then $\nabla_s=\d+\omega+r\cdot a$ and therefore
\begin{equation}
F_s=r^2(a\wedge a)+r(da+\omega \wedge a)+(d\omega+\omega\wedge \omega).
\end{equation}

So, similarly, we have
\begin{rmk}
When the twisted part $E$ is vector bundle of rank $k>1$, the 1-form $\hat{a}$ is $\mathfrak{u}(k)$-valued and the leading order term in this part is given by
\begin{equation}
(\frac{1}{2\pi \sqrt{-1}})^{\frac{n+1}{2}} \frac{1}{(\frac{n+1}{2})!} r^{n}\int_M \tr{a^n}.
\end{equation}
\end{rmk}

\section{The estimate of $\eta$-Invariant}

The second part of the estimate concerns the asymptotic of $\eta$-invariant i.e. $\eta(D_1)=\eta(D+c(\hat{a}))$. As mentioned before, we $\eta$-invariant is related with heat operator by
\begin{equation}\label{eta}
\eta(D_1)=\frac{1}{\sqrt{\pi}}\int_0^\infty t^{-\frac{1}{2}} \Tr{D_1e^{-tD_1^2}}.
\end{equation}
And we separate the estimate into two parts as follows,
\begin{equation}\label{etasp}
\eta(D_1) = \frac{1}{\sqrt{\pi}}\left(\int_{t_0}^\infty t^{-\frac 12} \Tr{D_1e^{-tD_1^2}}\d t +\int_0^{t_0} t^{-\frac 12} \Tr{D_1e^{-tD_1^2}}\d t\right).
\end{equation}
The choice of the constant $t_0$ depends on the parameter $r$. To be more precise we define another parameter $R=\sup_M \{\abs{F_1}\}$. Indeed its dependence on $r$ is

\begin{equation}\label{R}
R \sim \begin{cases}
r, &\text{for }k=1;\\
r^2 &\text{for } k>1.
\end{cases}
\end{equation}
where $k$ is the rank of the vector bundle $E$ and $t_0$ is chosen as $t_0=\frac{1}{2R}$.

\subsection{The estimate of the large-time part}

We start the estimate with $\int_{t_0}^{+\infty}t^{-\frac{1}{2}} \Tr{D_1 \exp{(-tD_1^2)}}\d t$. First of all,  the integral $\int_{t_0}^\infty t^{-\frac 12} \Tr{D_1\exp{(-tD_1^2)}}\d t$ can be controlled by the following observation.
\begin{lem}
Given any $t_0>0$, it follows that
\begin{equation}
\abs{\int_{t_0}^\infty t^{-\frac 12} \Tr{D_1\exp{(-tD_1^2)}}\d t} 
\leqslant\frac{\sqrt{\pi}}{2} \Tr{\exp{(-\frac{t_0}{2}D_1^2})}. 
\end{equation}
\end{lem}

\begin{proof}
For all $t\geqslant t_0$, from H\"older's inequlity, it follows that
\begin{equation}
\begin{split}
\abs{\Tr{D_1\exp{(-t D_1^2)}}}\leqslant &\norm{D_1\exp{(-\frac{t}{2}D_1^2)}}\abs{\Tr{\exp{(-\frac{t}{2}D_1^2)}}}\\
\leqslant &\norm{D_1\exp{(-\frac{t}{2}D_1^2)}}\abs{\Tr{\exp{(-\frac{t_0}{2}D_1^2)}}}.
\end{split}
\end{equation}

Noticing that, for all $\lambda\in \mathbb{R}$,
\begin{equation}
\int_{t_0}^\infty t^{-\frac{1}{2}} \abs{\lambda e^{-\frac{t \lambda^2}{2}}} \d t< \frac{\sqrt{\pi}}{2},
\end{equation}
it then follows that
\begin{equation}
\int_{t_0}^\infty t^{-\frac{1}{2}}\norm{D_1\exp{(-\frac{t}{2}D_1^2)}}\d t\leqslant \frac{\sqrt{\pi}}{2}.
\end{equation}
And therefore
\begin{equation}
\abs{\int_{t_0}^\infty t^{-\frac 12} \Tr{D_1\exp{(-tD_1^2)}}\d t}\leqslant\frac{\sqrt{\pi}}{2} \abs{\Tr{\exp{(-\frac{t_0}{2}D_1^2})}}.
\end{equation}
\end{proof}

The estimate for the kernel $K_1(t;x,y)$ of the heat operator $\exp{(-tD_1^2)}$ is given by Taubes in \cite{T1}, which is stated as follows will provide the first-step estimate of $\Tr{\exp{(-\frac{t_0 }{2}D_1^2})}$.

\begin{prop}
There exists a constant $C>0$, such that for all $t>0$ and $x,y\in M$,
\begin{equation}
\abs{K_{1}(t,x,y)}< C t^{-\frac n2} e^{C R t}e^{-\frac{d^2(x,y)}{4t}}
\end{equation}
\end{prop}

\begin{proof}
Fix an arbitrary $y\in M$, we have
\begin{align*}
\frac{\partial}{\partial t}\abs{K_1(t;x,y)} 
					    =&\inpro{K_1(t;x,y),K_1(t;x,y)}^{-1/2}\inpro{\frac{\partial}{\partial t}K_1(t;x,y),K_1(t;x,y)}_x^{1/2}\\
					    =&-\abs{K_1(t;x,y)}^{-1}\inpro{D_1^2K_1(t;x,y),K_1(t;x,y)}_x^{1/2}\\
					    =&-\abs{K_1(t;x,y)}^{-1}(\inpro{\nabla_1^\ast \nabla_1 K_1(t;x,y),K_1(t;x,y)}_x\\
					    &+\inpro{(\frac{K}{4}+c(F_1))K_1(t;x,y),K_1(t;x,y)}.
\end{align*}

On the other hand, for $\abs{K_1(t;x,y)}$, as a smooth function in the variable $x\in M$,

\begin{align*}
\d^\ast \d\abs{K_1(t;x,y)}
=&\abs{K_1(t;x,y)}^{-1}\left(\inpro{\nabla_1^\ast \nabla_1 K_1(t;x,y),K_1(t;x,y)}\right.\\
				      	&\left.+2\abs{\d \abs{K_1(t;x,y)}}^2-2\abs{\nabla_1 K_1(t;x,y)}^2\right)\\
					\leqslant&\abs{K_1(t;x,y)}^{-1}\inpro{\nabla^\ast \nabla K_1(t;x,y),K_1(t;x,y)},
\end{align*}
where the last line follows from Kato's inequality \eqref{katogr}.

There thus exists a constant $C>0$ such that, for all $x,y\in M$,
\begin{equation}\label{inq}
\frac{\partial}{\partial t}\abs{K_1(t;x,0)}\leqslant -\d^\ast \d\abs{K_1(t;x,y)} +CR\inpro{K_1(t;x,y),K_1(t;x,y)}_x^{1/2}),
\end{equation}
Now set $f(t;x)=e^{-C(1+R)t}\abs{K_1(t;x,y)}$, and then
\[
\frac{\partial}{\partial t}f(t;x)\leqslant -\d^\ast \d f(t;x).
\]

Denote by $h(t;x,y)$ the scalar heat kernel over $\mathbb{R}_+\times M \times M$, and recall its estimate in \cite{LY86} saying that
\begin{equation}
h(t;x,y)\leqslant C_0 t^{-\frac{n}{2}}\exp{(-\frac{d^2(x,y)}{4t})},
\end{equation}
for some constant $C_0>0$.
From the fact
\[
\lim_{t\rightarrow 0}K_1(t;x,y)=\delta_{y}(x)\mathrm{Id}_{x},
\]
it follows that
\begin{equation}
\lim_{t\rightarrow 0} f(t;x)-h(t;x,y)=0.
\end{equation}
Furthermore, along with that
\begin{equation}
(\frac{\partial}{\partial t}+\d^\ast \d)(f(t;x)-h(t;x,y))\leqslant 0,
\end{equation}
it finally follows from the maximal principle that there exists $C>0$ such that, for all $t>0$ and $x,y\in M$,
\begin{equation}
\abs{K_1(t;x,y)}\leqslant C t^{-\frac{n}{2}}e^{C R t}e^{-\frac{d^2(x,y)}{4t}}.
\end{equation}
\end{proof}

From our choice of $t_0=\frac{1}{2R}$, the estimate of this part then follows.
\begin{prop}\label{etast}
There exists a constant $C>0$ independent of $r$, such that
\begin{equation}
\abs{\Tr{\exp{(-\frac{t_0}{2}D_1^2)}}}\leqslant CR^{\frac{n}{2}}.
\end{equation}
\end{prop}

It now follows that there exists $C_1>0$ such that
\begin{equation}
\int_{t_0}^\infty t^{-\frac 12} \Tr{D_1\exp{(-tD_1^2)}}\d t \leqslant C_1 t_0^{-\frac n2} e^{C_1 R t_0}.
\end{equation}
This induces the estimate of the large-time part for the reduced $\eta$-invariant.

\begin{thm}\label{lgt}
There exists a constant $C>0$, such that
\[
\abs{\mathrm{dimKer}(D_1)+\int_{t_0}^\infty t^{-\frac{1}{2}}\Tr{D_1 \exp{(-tD_1^2)}}\d t} \leqslant CR^{\frac{n}{2}}
\]
\end{thm}

For the completeness we prove the Kato's inequality involved in the previous proof.
\begin{lem}[Kato's Inequality]
Given a unitary connection $\nabla$ on $S\otimes E$ over $M$, for any vector field $X\in \Gamma(TM)$ and an arbitrary $s\in \Gamma(S\otimes E)$, one has
\begin{equation}\label{katovct}
\abs{X\abs{s}}\leqslant \abs{\nabla_X s},
\end{equation}
in other words,
\begin{equation}\label{katogr}
\abs{\d\abs{s}}\leqslant \abs{\nabla s}.
\end{equation}
\end{lem}

\begin{proof}
Noticing
\begin{equation*}
X(\abs{s}^2)=2\abs{s} (X \abs{s})=2\inpro{\nabla_X s, s},
\end{equation*}
it then follows that
\begin{equation*}
X \abs{s}=\frac{\inpro{\nabla_X s, s}}{\abs{s}}.
\end{equation*}
Therefore, we have
\[
\abs{X\abs{s}}\leqslant \abs{\nabla_X s},
\]
which implies that
\[
\abs{\d\abs{s}}\leqslant \abs{\nabla s}.
\]
\end{proof}

\subsection{Localization of the problem}
In this section and the ones that follows, we establish an estimate for the kernel of the operator $D_1 \exp{(-tD_1^2)}$ for small time, $t\in(0,t_0]$, paying special attention to its dependence on the parameter $r>0$. It has been pointed out in \cite{BF} that its asymptotic expansion starts with the term $t^\frac{1}{2}$, which  guarantees the convergence of the integral in \eqref{eta}. However, we need to extract uniform information of its asymptotic behavior on the parameter $r$, and, for this purpose, apply Getzler rescaling. This kind of estimate has been done in \cite{BL} and \cite{DLM} in different contexts. Combining the approaches together gives us the desired estimate.

In order to apply further technique, we first need a localization of the problem.  To be more precise, it is to decompose the heat kernel of $D_1^2$ into one piece supported sufficiently closed to the diagonal $\Delta=\{(x,x): x\in M\}\subseteq M \times M$ and the other piece which will turn out to be negligible for $t\in (0,\frac{1}{2R}]$. As in \cite{cheeger1987} and \cite{DLM}, the key to localize the problem is the finite propagation speed technique.

For any $\delta>0$, let $f_\delta: \mathbb{R} \rightarrow [0,1]$ be a smooth even function such that:
\begin{equation}\label{cutoff}
f_\delta (v)=
\begin{cases}
1 & \text{for } \abs{v}\leqslant \frac{\delta}{2}\\
0 & \text{for } \abs{v}>\delta,
\end{cases}
\end{equation}
and define:

\begin{defi}
For $t>0$, set
\begin{equation}
\begin{split}
G_t(D_1) &=\frac{1}{2\sqrt{\pi t}}\int_{-\infty}^{+\infty} e^{\sqrt{-1}vD_1}e^{-{v^2}/{4t}}f_\delta(v)\d v \\
H_t(D_1) &=\frac{1}{2\sqrt{\pi t}}\int_{-\infty}^{+\infty} e^{\sqrt{-1}vD_1}e^{-{v^2}/{4t}}(1-f_\delta(v))\d v,
\end{split}
\end{equation}
and let $G(t;x,y)$ and $H(t;x,y)$ be the smooth kernels associated to $G_t(D_1)$ and $H_t(D_1)$.
\end{defi}

Clearly,
\begin{equation}
G_t(D_1)+H_t(D_1)=\exp{(-tD_1^2)},
\end{equation}
and, accordingly,
\begin{equation}
G(t;x,y)+H(t;x,y)=K_{1}(t;x,y).
\end{equation}

Firstly, it can be shown that for $x,y \in M$, $G(t;x,y)$ only depends on the restriction of $D_1$ to $B^M(x,\delta)$, i.e. $G(t;x,y)=0$ if $d(x,y)>\delta$. A sketch is presented as follows.

Noticing that  for any fixed $u>0$, 
\begin{equation}
\begin{split}
G_t(D_1) &=\int_{-\infty}^{+\infty} e^{-\frac{v^2}{4t}}f_\delta(v)\exp{(\sqrt{-1}v\sqrt{u}D_1)}\d v\\
&=\int_{-\delta}^{\delta} e^{-\frac{v^2}{4t}}f_\delta(v)\cos{(v D_1)}\d v.
\end{split}
\end{equation}
Take an arbitrary smooth section $\sigma \in \Gamma(S \otimes E)$ such that $\mathrm{supp}(\sigma)\subseteq  B^M(x,\rho)$. It follows from the finite propagation-speed property of the wave operator $\cos(vD_1)$ (c.f. \cite{CNF}) that, $\mathrm{supp}\{\cos(\sqrt{-1}vD_1)\sigma\}\subseteq B^M(x,\rho+\abs{v})$. Therefore the operator $G_t(D_1)\sigma$ is supported within the open ball $B^M(x,\rho+\delta)$. In terms of its kernel, it means that $G(t;x,y)$ only depends on the restriction of $D_1$ to $B^M(x,\delta)$, and is zero if $d(x,y)>\delta$.

It now remains to show the other part, the kernel of $H_t(D_1)$ is negligible in the following sense. We first define the $\mathcal{C}^m$-norm for a smooth section $\sigma \in \Gamma(S \otimes E)$ by
\begin{equation}
\abs{\sigma}_{\mathcal{C}^m}=\sup \{ \abs{\nabla_{0,v_1} \nabla_{0,v_2}, \cdots , \nabla_{0,v_m}\sigma }: v_i \in T_x M,\text{ and } \abs{v_i}=1,  \forall i=1,2,\dots,m, x\in M\},
\end{equation}
and accordingly, for the kernel $H(t;\cdot, \cdot)$, it induces the $\mathcal{C}^m$-norm as
\begin{equation}
\abs{H(t;\cdot, \cdot)}_{\mathcal{C}^m}=\sup \{ \abs{\nabla_{0,v_1}^{(1)} \nabla_{0,v_2}^{(1)}, \cdots , \nabla_{0,v_j}^{(1)}\nabla_{0,u_1}^{(2)} \nabla_{0,u_2}^{(2)}, \cdots , \nabla_{0,u_j}^{(2)}H(t;x,y)}: i+j=m\},
\end{equation}
where $v_p$'s and $u_q$'s are unit vectors in $T_x M$ and $T_y M$ respectively.
\begin{prop}\label{expdc}
For any $m \in \mathbb{N}$, $\delta >0$, there exist $\alpha, C_m>0$ such that for all $t>0$,
\begin{equation}
\abs{H(t;\cdot, \cdot)}_{\mathcal{C}^m}\leqslant C_m r^{n+m+2}e^{-{\alpha \delta^2}/{t}}.
\end{equation}
\end{prop}
\begin{rmk}
\begin{enumerate}

\item Indeed this proposition shows that the kernel $H(u;\cdot,\cdot)$ decays exponentially as $r \rightarrow +\infty$. Note that $r^{n+m+1} e^{-\alpha \delta^2 R}$ is uniformly bounded for all $r>0$ according to \eqref{R}. It then follows that, for all $m\in \mathbb{N}$, there exists a constant $C_m'>0$ such that for all $t\in (0,\frac{1}{2R}]$,
\begin{equation}\label{expdcst}
\abs{H(t;\cdot, \cdot)}_{\mathcal{C}^m}\leqslant C_m' e^{-{\alpha \delta^2}/{2t}}.
\end{equation}

\item This decomposition also works for the localization for the kernel of the operator $D_1\exp{(-uD_1^2)}$ and allows us the do the estimating in the following sections.

\item The proof is similar to \cite{DLM}, and for completeness we present some details here. The basic idea is obtain the $L^2$-estimate of the kernel function via the Fourier synthesis. Then one improves to the $C^m$-estimate using the Sobolev inequality.
\end{enumerate}
\end{rmk}

Denote the function $(1-f_\delta(v))e^{-v^2/4t}$ by $\hat{h}_t(v)$ and, as the first step, we show the following estimate, which is an analogue of \cite[Proposition 1.1]{CGT}.
\begin{prop}\label{estfouriersynp}
For any $k\in\mathbb{N}$, and $\sigma \in \Gamma (S\otimes E)$,
\begin{equation}\label{estfouriersyn}
\norm{D_1^k H_t(D_1)\sigma}_{L^2}\leqslant \int_{\frac{\delta}{2}}^\infty \abs{\hat{h}^{(k)}_t(v)}\d v.
\end{equation}
\end{prop}
\begin{proof}
From the definition of the operator $H_t(D_1)$, it follows that
\begin{equation}
D_1^kH_t(D_1) \sigma =\frac{1}{2\sqrt{\pi t}}\int_{-\infty}^{\infty} (\sqrt{-1})^{-k} \hat{h}_t^{(k)}(v)\exp{(\sqrt{-1}vD_1)}\sigma \d v.
\end{equation}

Note that, by the energy estimate of wave equations, $\norm{\exp{(\sqrt{-1}vD_1)}\sigma}_{L^2}\leqslant \norm{\sigma}_{L^2}$.
On the other hand, since $\hat{h}_t \equiv 0$ over $[0,\frac{\delta}{2}]$, so is its derivatives, and therefore
\begin{equation}
\begin{split}
\norm{D_1^kH_t(D_1) \sigma}_{L^2}
&\leqslant \frac{1}{2\sqrt{\pi t}}\norm{\exp{(\sqrt{-1}vD_1)}\sigma}_{L^2}\cdot \int_{-\infty}^{\infty} \abs{\hat{h}_t^{(k)}(v)}\d v\\
&\leqslant \frac{1}{\sqrt{\pi t}}\norm{\sigma}_{L^2}\int_{\frac{\delta}{2}}^{\infty} \abs{\hat{h}_t^{(k)}(v)}\d v.
\end{split}
\end{equation}
\end{proof}

To improve the estimate to $L^\infty$-estimate of the kernel function $H(t;x,y)$, we make use of the Sobolev inequality, which entails estimating the $L^2$-norms of derivatives of $H(t;x,y)$ in the variables $x$, $y$. Here we introduce a Sobolev norm for smooth sections of the bundle $S\otimes E$ as \cite{roe1988}.

\begin{defi}
Define the $m$-th Sobolev norm for a smooth section $\sigma \in \Gamma(S \otimes E)$ by
\begin{equation}
\norm{\sigma}_m = \sum_{i=0}^m \norm{D_0^i \sigma}_{L^2}.
\end{equation}
\end{defi}

We then have the following result which relates \eqref{estfouriersyn} with this Sobolev norm.
\begin{lem}\label{dtoD}
For all $m\in \mathbb{N}$, there exists a constant $C_{m}$ such that
\begin{equation}
\norm{\sigma}_{m} \leqslant C_{m} \sum_{j=0}^{m} r^{m-j} \norm{D_1^j \sigma}_{L^2}.
\end{equation}
\end{lem}
\begin{proof}
Since, by its definition, $D_1=D_0+rc(a)$, for some $\End(E)$-valued one-form $a$, it is straightforward to see that, for any $\sigma\in \Gamma(S\otimes E)$,
\begin{equation}
\norm{\sigma}_1\leqslant C_1(\norm{D_1 \sigma}_{L^2}+r\norm{\sigma}_{L^2}),
\end{equation}
for some $C_1>0$.

Furthermore, note that
\begin{equation}
[D_1, D_0^m]=[D_0, D_0^m]+r[c(a),D_0^m]=r[c(a),D_0^m],
\end{equation}
which is a differential operator of order at most $m$. There thus exists a $C_1'>0$, such that

\begin{equation}
\begin{split}
\norm{D_0^m \sigma}_{H^1} &\leqslant C_1(\norm{D_1 D_0^m \sigma}_{L^2}+r\norm{D_0^m \sigma}_{L^2})\\
								&\leqslant C_1(\norm{D_0^m D_1 \sigma}_{L^2}+r\norm{[c(a), D_0^m]\sigma}_{L^2}+r\norm{D_0^m \sigma}_{L^2})\\
								&\leqslant C_1'(\norm{D_1 \sigma}_{H^m}+r\norm{\sigma}_{H^m}).
\end{split}
\end{equation}
This implies that
\begin{equation}
\norm{s}_{H^m}\leqslant C_m(\norm{D_1 s}_{H^{m-1}}+r\norm{s}_{H^{m-1}}).
\end{equation}
And, by repeating this argument, we have 
\begin{equation}
\begin{split}
\norm{\sigma}_{H^m}	&\leqslant C_1'(\norm{D_1 \sigma}_{H^{m-1}}+r\norm{\sigma}_{H^{m-1}})\\
								&\leqslant C_1'((\norm{D_1^2 \sigma}_{H^{m-2}}+r\norm{D_0^m \sigma}_{H^{m-2}})+r(\norm{D_1\sigma}_{H^{m-2}}+r\norm{\sigma}_{H^{m-2}}))\\
								&\leqslant \dots \leqslant C_m \sum_{j=0}^m r^{m-j}\norm{D_1^j \sigma}_{L^2}.
\end{split}
\end{equation}
\end{proof}

Actually, Lemma \ref{dtoD} and Proposition \ref{estfouriersynp} together, provides the estimate for the Sobolev norm of the kernel via the Fourier synthesis, and now we apply this to show the desired estimate.

\begin{proof}[Proof of Proposition \ref{expdc}]

Given any $k,l \in \mathbb{N}$, and an arbitrary section $\sigma \in \Gamma (S \otimes E)$, it then follows from Lemma \ref{dtoD} that
\begin{equation}
\norm{D_0^kH_t(D_1)D_0^l \sigma}_{L^2} \leqslant C_k \sum_{j=0}^k r^{k-j} \norm{D_1^j H_t(D_1)D_0^l \sigma}_{L^2}.
\end{equation}
And once again, for each $0\leqslant j \leqslant k$ the adjoint of the operator $D_1^j H_t(D_1)D_0^l$ is $D_0^l H_t(D_1)D_1^j=D_0^l D_1^j H_t(D_1)$ and
\begin{equation}
\norm{D_0^l D_1^j H_t(D_1)\sigma}_{L^2} \leqslant C_l \sum_{m=0}^{l} r^{l-m} \norm{D_1^{j+m} H_t(D_1)\sigma}_{L^2},
\end{equation}
for some $C_l>0$, which implies that 
\begin{equation}
\norm{D_1^j H_t(D_1)D_0^l \sigma}_{L^2}\leqslant C_l\sum_{m=0}^{l} r^{l-m}  \norm{D_1^{j+m} H_t(D_1)\sigma}_{L^2}.
\end{equation}

Combining these inequalities, there then exists a $C_{k,l}>0$, such that
\begin{equation}
\norm{D_0^k H_t(D_1)D_0^l \sigma}_{L^2} \leqslant C_{k,l} \sum_{j=0}^{k+l} r^{k+l-j} \norm{D_1^j H_t(D_1)\sigma}_{L^2}.
\end{equation}
On the other hand, by \eqref{estfouriersyn}, we have
\begin{equation}
\norm{D_1^jH_t(D_1)}_{L^2}\leqslant \frac{1}{\sqrt{\pi t}} \int_{\frac{\delta}{2}}^{\infty} \abs{\hat{h}_t^{(j)}(v)}\d v.
\end{equation}
In addition, according to \cite{CGT}, there exists an $\alpha>0$ such that for any $j>0$, 
\[
\abs{\hat{h}_t^{(j)}(v)}\leqslant C_j t^{-\frac{j+1}{2}} e^{-\frac{\alpha v^2}{t}},
\]
for some $C_j>0$, and therefore,
\[
\int_{\frac{\delta}{2}}^{\infty} \abs{\hat{h}_t^{(j)}(v)} \leqslant C_j' e^{-\frac{\alpha \delta^2}{t}},
\]
for some $C_j'>0$. It then follows that for any $k,l \in \mathbb{N}$, and an arbitrary $\sigma \in \Gamma (S\otimes E)$, there exists a constant $C_{k,l}>0$,
\begin{equation}\label{sblnmop}
\norm{D_0^k H_t(D_1) D_0^l \sigma}_{L^2} \leqslant C_{k,l} r^{k+l} e^{-\frac{\alpha \delta^2}{t}}.
\end{equation}

Note that, by definition of the kernel $H(t;x,y)$, \eqref{sblnmop} can be rewritten as
\begin{equation}
\norm{D_0^k \int_M H(t;x,y) D_0^l\sigma(y) \d y}_{L^2}  \leqslant C_{k,l} r^{k+l} e^{-\frac{\alpha \delta^2}{t}}\norm{\sigma}_{L^2}.
\end{equation}
Then, from Sobolev inequality, it follows that, as smooth section of $S \otimes E$,
\begin{equation}
\norm{\int_M H(t;x,y)D_0^l \sigma(y) \d y}_{\mathcal{C}^k}\leq C_{m,l} r^{k+l+\frac{n+1}{2}}e^{-\frac{\alpha \delta^2}{t}} \norm{\sigma}_{L^2},
\end{equation}
for some $C_{m,l}$, i.e., for all $x\in M$ and tangent vectors $v_1, v_2, \dots , v_m$,
\begin{equation}
\abs{\int_M \nabla_{0,v_1}^{(1)} \nabla_{0,v_2}^{(1)} \dots \nabla_{0,v_k}^{(1)} H(t;x,y)D_0^l \sigma(y) \d y}\leqslant C_{k,l} r^{n+k+l+1}e^{-\frac{\alpha \delta^2}{t}} \norm{\sigma}_{L^2}.
\end{equation}
Then since the section $\sigma$ is arbitrary, it follows for any fixed $x\in M$ that
\begin{equation}
\left( \int_M \abs{\nabla_{0,v_1}^{(1)} \nabla_{0,v_2}^{(1)} \dots \nabla_{0,v_k}^{(1)} D_0^{(2),l} H(t;x,y)}^2 \d y \right)^{\frac{1}{2}}\leqslant C_{k,l} r^{k+l+\frac{n+1}{2}}e^{-\frac{\alpha \delta^2}{t}}.
\end{equation}
Then, by applying Sobolev inequality again, there exists a $C_m>0$, such that, for all $t>0$,
\begin{equation}
\abs{H(t;\cdot, \cdot)}_{\mathcal{C}^m} \leqslant C_m r^{n+m+1} e^{-\frac{\alpha \delta^2}{t}}.
\end{equation}
\end{proof}

\subsection{A local coordinate system and a trivialization of $S \otimes E$}
Now denote by $\delta$ the injective radius of the manifold $M$. Fixing an arbitrary $x\in M$, denote by $B^M(x,\delta)$ and $B^{T_x M}(0,\delta)$ the open balls in the manifold $M$ and the tangent space $T_x M$ respectively, which are both centered at the point $x$ and of radius $\delta$. Then the exponential map $\exp_x: B^{T_x M}(0,\delta) \rightarrow B^M(x,\delta)$ identifies them with a diffeomorphism. 

Furthermore, and take an $\epsilon\in (0,\frac{\delta}{4})$ such that the restriction of the vector bundles $S$ and $E$ are both trivial over the open ball $B^M (x, \epsilon)$. Then, let $\{e_i\}_{i=1}^n$ be an oriented orthonormal basis of $T_{x}M$. We also denote by $\{e^i\}_{i=1}^n$ the dual basis of ${e_i}$. Given $X\in T_x M$, let $\tilde{e}_i(X)$ be the parallel transport of $e_i$ with respect to $\nabla^{TM}$ along the curve $\gamma_X: [0, 1] \ni u \rightarrow \exp^M_{x}(uX)$. Identify $S_X, \, E_X$ for $X \in B^{T_{x}M}(0, \varepsilon)$ to $S_{x}$ and $E_{x}$ by parallel transport with respect to the connections $\nabla^S$ and $\nabla^E$ along the above curve. 
Furthermore, we extend the geometric objects on $B^{T_{x}M}(0, \varepsilon)$ to $\mathbb{R}^{n} \simeq  T_{x}M$ (here we identify $(X_1,\dots ,X_n) \in \mathbb{R}^{n}$ to $\sum_i X_i e_i \in T_{x}M$ such that $D_1$ coincides with a Dirac operator on $\mathbb{R}^n$ associated to a Hermitian vector bundle within this ball, such that we can replace $M$ by $\mathbb{R}^n$. Such a Dirac operator $\D_1^x$ is constructed in the following way.

To be more precise, we define as in \cite{DLM} a function $\varphi_\epsilon: \mathbb{R}^n \rightarrow \mathbb{R}^n$ by 
\begin{equation}
\varphi_\epsilon(X):= f_{2\epsilon}(\abs{X})X,
\end{equation}
and equip $\mathbb{R}^n$ with the metric
\begin{equation}
g^x\vert_X= \exp_x^\ast(g_0)\vert_{\varphi_\epsilon(X)},
\end{equation}
where $g_0$ is the standard Euclidean metric on $\mathbb{R}^n$. Denote by $k^x$ the scalar curvature of $(\mathbb{R}^n,g^x)$.

We denote this trivial bundle over $\mathbb{R}^n$ by $S_0\otimes E_0$, and equip the trivial vector bundle $S_0$ and $E_0$ with the connection 
\begin{equation}
\nabla^{x,S_0}:=(\exp_x\circ\varphi_\epsilon)^\ast(\nabla^S) \qquad \text{and} \qquad \nabla_1^{x,E_0}:=(\exp_x \circ \varphi_\epsilon)^\ast(\nabla_1^E)
\end{equation}
respectively. Set 
\[
\nabla_1^x=\nabla^{x,S_0}\otimes 1 +1\otimes \nabla_1^{x,E_0},
\]
and denote by $\F^x_1$ the curvature of the connection $\nabla_1^{x,E_0}$.

We now identify $C_c^\infty (\mathbb{R}^n, S_x \otimes E_x)$ with the space of compactly supported sections of the bundle $S_0\otimes E_0$. Let $\D^x_{1}=\sum
_{i=1}^n c(e_i) \nabla^x_{1,\tilde{e}_i}$ be the induced Dirac operator acting on $C_c^\infty (\mathbb{R}^n, S_x \otimes E_x)$ and denote by $\P^1_x(t;X,Y)$ the smoothing kernel associated to the operator $\D_{1}\exp{(-t \D_{1}^2)}$. We also denote by $P(t;x,y)$ the smoothing kernel of the operator $D_{1}\exp{(-t D_{1}^2)}$ on the manifold $M$. Then one has from the previous estimates the following result that allows us to reduce the estimate to that on a trivial bundle over $\mathbb{R}^n$.

\begin{cor}\label{estlocl}
There exist constants $C,\alpha>0$ such that, for all $t\in (0,\frac{1}{2R}]$,
\begin{equation}
\abs{P(t; x,x)-\P^1_x(t;0,0)}\leqslant C e^{-\frac{\alpha \delta^2}{t}}.
\end{equation}
\end{cor}

\subsection{Auxiliary Grassmann variable and an exponential Transformation}
Following \cite{BF} we introduce an auxiliary Grassmann variable to adapt the operator $\D_1^x \exp{(-u\D_{1}^{x,2})}$ into the context of a heat operator. As in \cite{BF}, we introduce a auxiliary Grassmann variable $z$ which anticommutes with the Clifford actions of tangent vector fields of $M$. It allows us to write the operator
\[
\exp{(-t\D_{1}^{x,2}+z\sqrt{t}\D_{1}^x)}=\exp{(-t\D_{1}^{x,2})}+z\sqrt{t}\D_{1}^x \exp{(-t\D_{1}^{x,2})}
\]
whose kernel is
\begin{equation}\label{zkernel}
\P_{x,z}(t;X,Y)=\P_x^0 (t;X,Y)+z\sqrt{t}\P^1_x(t;X,Y).
\end{equation}
Here $\P_x^0(t;X,Y)$ is the heat kernel and $\P_x(t;X,Y)$ is as above. Thus, for this kind of operators, we define $\Trz{\cdot}$ as the following.

\begin{defi}
For $A$, $B\in \mathrm{End}(S\otimes E)$, we define
\begin{equation}
\Trz{A+zB}=\Tr{B}.
\end{equation}
\end{defi}

And define similarly for the smoothing kernel that
\begin{equation}
\trz{\P_{x,z}(t;X,X)}=t^{\frac{1}{2}}\tr{\P^1_x(t;X,X)}.
\end{equation}
Bismut and Freed shown in \cite{BF} that there is a $C^\infty$ function $b_{1/2}(x)$ on $M$ such that as $t\rightarrow 0$,
\begin{equation}
\tr{\P^1_x(t;X, X)} = b_{1/2}(X)t^{1/2} + O(t^{3/2}),
\end{equation}
and $O(t^{3/2})$ is uniform on $M$. For this problem, we need in particular to look further for its uniform dependence on the parameter $r$.

In order to do further estimate, we introduce a map $\mathcal{R}^c: C^\infty (\mathbb{R}^n , S_x \otimes E_x) \rightarrow C^\infty (\mathbb{R}^n ,S_x\otimes E_x)$ defined by
\begin{equation}
(\mathcal{R}^c \sigma)(X)=\sum_{i=1}^n X_ic(e_i)\sigma(X),
\end{equation}
and define an exponential map by
\begin{equation}
\mathcal{E}=\exp{(-\frac{z\mathcal{R}^c}{2\sqrt{t}})}=I-\frac{z\mathcal{R}^c}{2\sqrt{t}}.
\end{equation}
Conjugating the operator $\exp(-\D_1^{x,2}+z\sqrt{t}\D_1^{x})$ by $\mathcal{E}$, one then has
\begin{equation}\label{cgtdopt}
\begin{split}
&\mathcal{E}^{-1}\exp{(-t \D_{1}^{x,2}+z\sqrt{t}\D_{1}^x)}\mathcal{E}\\
=&\exp{(-t \D_{1}^{x,2})}+z \left(t^{\frac{1}{2}}\D_{1}^x \exp{(-t \D_{1}^{x,2})}-\frac{1}{2\sqrt{t}}[\exp{(-t \D_{1}^{x,2})},\mathcal{R}^c]\right).
\end{split}
\end{equation}

In particular, the kernel of the operator $\frac{1}{2}[\exp{(-t \D_{1,t}^{x,2})},\mathcal{R}^c]$ is
\begin{equation}
\P'_x(t;X,Y)=\frac{1}{2}\left(\P_x^0(t;X,Y)\sum_{i}Y_i c(e_i) -\sum_{i}X_i c(e_i) \P_x^0(t;X,Y)\right).
\end{equation}
It then follows immediately that
\begin{equation}
\P'_x(t;0,0)=0.
\end{equation}

For the rest of this paper, we denote by
\begin{equation}
\P^2_x(t;X,Y)=t^{\frac{1}{2}}\P^1_x(t;X,Y)-t^{-\frac{1}{2}}\P'_x(t;X,Y)
\end{equation}
the smoothing kernel of the operator
\begin{equation}
\L^x(t)=t^{\frac{1}{2}}\D_{1}^x \exp{(-t \D_{1}^{x,2})}-\frac{1}{2\sqrt{t}}[\exp{(-t \D_{1}^{x,2})},\mathcal{R}^c].
\end{equation}
And it follows, for all $t>0$, that
\begin{equation}\label{treq}
\tr{\P^2_x(t;0,0)}=t^{\frac{1}{2}}\tr{\P^1_x(t;0,0)}.
\end{equation}

\subsection{Gezler's rescaling and a family of Sobolev spaces}
According to the previous discussion, the estimate remains now to be dealt with is that of the kernel $\P^2_x(t,\cdot,\cdot)$ in \eqref{treq} for $t\in(0, \frac{1}{2R})$, where the variable $t$ is small while the parameter $r$ is taken large. As in \cite{BL}, here we approximate the trace by applying Getzler's rescaling technique.

As in the previous section, denote by $\gs{t}$ the Getzler's rescaling acting on $C_c^\infty(\mathbb{R}^n, S_x\otimes E_x)$, and set for $t\in (0,\frac{1}{2R})$
\begin{equation}
\D_{1,t}^x=t^{\frac{1}{2}}\g{t}{\D_1^x}.
\end{equation}

And denote by
\begin{equation}
\P_{x,t}^2(u;X,Y)=t^{\frac{n}{2}}\gs{t}\P_x^2(u;X,Y)
\end{equation}
the kernel of the rescaled operator
\begin{equation}\label{opltx}
\L_{t}^x(u):=\g{t}{\L^x(u)}=u^{\frac{1}{2}}\D_{1,t} \exp{(-u \D_{1,t}^{x,2})}+\frac{1}{2} u^{-\frac{1}{2}}[\exp{(-u \D_{1,t}^{x,2})},\mathcal{R}^c_t],
\end{equation}
where $\mathcal{R}^c_t$ is defined by
\begin{equation}
\mathcal{R}^c_t \sigma(X)=\sum_{i=1}^n  X_i c_t(e_i) \sigma(X), \qquad \forall \sigma \in C^\infty (\mathbb{R}^n, S_x\otimes E_x).
\end{equation}

Furthermore, if we write $\P^2_x(u;X,X)$ as 
\begin{equation}
\P^2_x(u;X,X)=\sum_{\abs{I} \text{ odd}}a_I(u;x)c(e_I),
\end{equation}
where each $a_I(u;X)$ is an endomorphism of the vector space $E_x$. Then 
\begin{equation}
\P_{x,t}^2(u;X,X)=t^{\frac{n}{2}}(\sum_{\abs{I} \text{ odd}}a_I(tu;t^\frac{1}{2} X)c_t(e_I)),
\end{equation}
can be written more precisely as
\begin{equation}
\P_{x,t}^2(u;X,X)=\sum_{I,J} A_I^J(u;X) \varepsilon(e_I) \iota(e_J),
\end{equation}
where $A_I^J \in \ed{E_x}$, $I=(i_1,i_2,\dots, i_p)$, $J=(j_1,j_2,\dots, j_q)$ are multi-indices with $1\leqslant i_1\leqslant i_2 \leqslant \dots\leqslant i_p\leqslant n$ and $1\leqslant j_1 \leqslant j_2 \leqslant \dots \leqslant j_q \leqslant n$.

Set
\begin{equation}
\left[\P_{x,t}^2(u;X,X)\right]^{\rm{max}}=A_{1,\dots, n}(tu;X) \in \ed{E_x}.
\end{equation}

The following result then follows immediately from the property of $\tr{c(e_I)}$.
\begin{prop}
The following identity holds
\begin{equation}\label{trvsrscl}
\tr{\P^2_x(tu;0,0)}=\frac{1}{2}(-2\sqrt{-1})^{\frac{n+1}{2}} \tr{\P_{x,t}^2(u;0,0)}^{\rm{max}}.
\end{equation}
\end{prop}

This allows us to establish the estimate of $\tr{\P^2_x(u;0,0)}$ via that of $\P_{x,t}^2(u;0,0)$, which the rest of this section will deal with. For this purpose, we will express $\exp{(-u\D_{1,t}^{x,2})}$ as an integral along a contour $\Gamma$ in the complex plane $\mathbb{C}$, so that
\begin{equation}
\exp{(-u\D_{1,t}^{x,2})}=\int_\Gamma \exp{(-u \lambda)} (\lambda- \D_{1,t}^{x,2})^{-1} \d\lambda.
\end{equation}

From now on, we identify $S_x$ with $\Lambda^\ast({T}_x^\ast M)$ and denote by $\mathbf{I}^0_{p,x}$ the  set of square integrable sections of $\Lambda^p({T}_x^\ast M)\otimes E_x$. And apply here a family of norm as introduced in \cite{BL}.
\begin{defi}
For $t\in (0, \frac{1}{2R}]$ and $\sigma \in \mathbf{I}^0_{p,x}$, set
\begin{equation}\label{l2nrm}
\norm{\sigma}_{x,t,0}^2=\int_{\mathbb{R}^n} \abs{\sigma(X)}^2 \left( 1+ f_\delta (t^\frac{1}{2} \abs{X}) \abs{X} \right)^{2(n-p)} \d X.
\end{equation}
\end{defi}

Then \eqref{l2nrm} induces a Hermitian product $\inppro{\cdot,\cdot}_{x,t,0}$ on $\mathbf{I}^0_{x} =\bigoplus \mathbf{I}^0_{p,x}$ with the direct sum of the Hermitian products. We will say that a family of operators acting on $\mathbf{I}^0_{x}$, which depend on $t\in(0,\frac{1}{2R}]$, is uniformly bounded if their norms calculated with respect to the norms $\norm{\cdot}_{x,t,0}$ are uniformly bounded. It is show in \cite[Proposition 11.24]{BL} that the Hilbert norms $\norm{\cdot}_{x,t,0}$ have been tailor-made for the Proposition which follows to be true.
\begin{prop}[{\cite[Proposition 11.24]{BL}}]\label{unbdactions}
The following families of operators 
\begin{equation}
\mathbf{1}_{B_{\delta/2\sqrt{t}}(0)}\cdot \abs{X}^k \cdot (\varepsilon(e_i)-t^2 \iota(e_i)), \quad i=1,2,\dots, n, \textit{and } k=0,1,
\end{equation}
acting on $\mathbf{I}^0_{x}$ and depending on $t\in (0,\frac{1}{2R}]$, are uniformly bounded.
\end{prop}

Furthermore, recall that, as in the Section 3, 
\begin{equation}
\nabla_{1,t}^x=t^{\frac{1}{2}}\delta_t \nabla_1^x \delta_{t}^{-1}=\nabla_t^{x,S_0}\otimes 1 + 1\otimes \nabla_{1,t}^{x,E_0}.
\end{equation}
And now we introduce a family of Sobolev norm in the following way.
\begin{defi}
For any $t\in (0,\frac{1}{2R}]$, $m\in \mathbb{N}^\ast$ and $\sigma \in C_c^\infty(\mathbb{R}^n, S_x \otimes E_x)$, set
\begin{equation}\label{snrm}
\norm{\sigma}_{x,t,m}^2=
\sum_{l=0}^{m}\sum_{i_1,\dots,i_l=1}^{n}\norm{\nabla_{1,t,\partial_{i_1}},\dots,\nabla_{1,t,\partial_{i_l}}^x \sigma}^2_{x,t,0}.
\end{equation}
\end{defi}

It then induces canonically a norm for the operators as follows.
\begin{defi}
If $A\in \mathscr{L}(\mathbf{I}_x^m,\mathbf{I}_x^{m'})$ for some integers $m,m'$, we denote by $\norm{A}^{m,m'}_t$ the norm of $A$ induced by $\norm{\cdot}_{t,m}$ and $\norm{\cdot}_{t,m'}$, i.e.
\begin{equation}
\norm{A}^{m,m'}_{x,t}=\sup_{\norm{\sigma}_{x,t,m}=1} \norm{A\sigma}_{x,t,m'}.
\end{equation}
\end{defi}

A fundamental question to ask here is the relation between this Sobolev norm with the ``usual" one and, in particular, its dependence on the parameters $r$ and $t$. Here by ``usual", it means for a smooth section $\sigma  \in C^\infty_c(\mathbb{R}^n, S_x \otimes E_x)$, the Sobolev norm of the fiberwise magnitude $\abs{\sigma}$ as a smooth function on $\mathbb{R}^n$, denoted by $\norm{\sigma}_m$ ($m\in \mathbb{N}$), and the question is answered by the following comparison.

\begin{lem}\label{rscnrm}
The Sobolev norm defined in \eqref{snrm} is equivalent to the usual Sobolev norm on any closed ball $B(0,q) \subseteq \mathbb{R}^n$ in the following sense. There exists a constant $C>0$, such that for all the parameters $r\geqslant1$ and $t\in (0,\frac{1}{2R}]$, given $\sigma\in C^\infty (\mathbb{R}^n, S_x\otimes E_x)$ with $\supp{\sigma}\subseteq B(0,q)$,
\begin{equation}
\frac{1}{C(1+q)^{m+n}}\norm{\sigma}_{x,t,m} \leqslant \norm{\sigma}_m \leqslant C(1+q)^{m}\norm{\sigma}_{x,t,m}.
\end{equation}
\end{lem}

\begin{proof}
Given any $\sigma\in C^\infty(\mathbb{R}^n, S_x\otimes E_x)$, such that $\supp{\sigma}\subseteq B(0,q)$, it follows directly from \eqref{l2nrm} that
\begin{equation}
\begin{split}
&\norm{\sigma}_{x,t,0}\geqslant \norm{\sigma}_{0},\\
&\norm{\sigma}_{x,t,0}\leqslant (1+q)^{n} \norm{\sigma}_{0}.
\end{split}
\end{equation}

In terms of the trivialization of the bundle $S_0$ and $E_0$, write the connection $\nabla^{x}_1$ as 
\[
\nabla^{x}_1=\d+\omega^x+\A_1^x,
\]
where $\omega^x$ and $\A_1^x$ are connection one-forms of $S$ and $E$ respectively, and, accordingly,
\[
\nabla^{x}_1=t^{\frac{1}{2}}\g{t}{\nabla_1}=\d+\omega_t^x+\A_{1,t}^x.
\]
It then follows as \eqref{cnctform1} that 
\begin{equation}\label{cnctform}
\begin{split}
&\A_{1,t}^x(\partial_i)\vert_X=t\int_0^1 \rho X^j \F_1^x(e_j,e_i)(\rho t^{\frac{1}{2}}X)\d\rho,\\
&\omega_t(\partial_i)\vert_X=t\int_0^1 \rho X^j R_t^{S}(\partial_j,\partial_i)(\rho t^{\frac{1}{2}}X)\d\rho,
\end{split}
\end{equation}
where $R_t^S(\partial_j,\partial_i)=\inpro{R^{TM}(\partial_j, \partial_i)\tilde{e}_k, \tilde{e}_l}c_t(e_k)c_t(e_l).$

There thus exists a $C_0>0$, such that, for all $X\in \mathbb{R}^n$,
\begin{equation}
\sum_{i=1}^n \abs{\nabla_{1,t,\partial_i} \sigma(X)}^2\leqslant C_0 \cdot  \abs{\abs{X}\sigma}^2+\abs{\d\abs{\sigma}}^2,
\end{equation}
and therefore,
\begin{equation}
\sum_{i=1}^n \abs{\nabla_{1,t,\partial_i} \sigma}^2\leqslant C_0 \cdot q^2 \abs{\sigma}^2+\abs{\d\abs{\sigma}}^2,
\end{equation}
while it follows from Kato's inequality \eqref{katovct} that
\begin{equation}
\sum_{i=1}^n \abs{\nabla_{1,t,\partial_i} \sigma}^2 \geqslant  \abs{\d\abs{\sigma}}^2.
\end{equation}

Therefore, there exists $C_1>0$, such that
\begin{equation}
\norm{\d\abs{\sigma}}^2 \leqslant \sum_{i=1}^n \norm{\nabla^x_{1,t,\partial_i}\sigma}_{x,t,0}^2 \leqslant C_1 q^2(1+q)^{2n} \norm{\sigma}_1.
\end{equation}
It then follows that
\begin{equation}
\norm{\sigma}_1\leqslant \norm{\sigma}_{x,t,1}\leqslant C_1 (1+q)^{n+1}\norm{\sigma}_1.
\end{equation}

By repeating this procedure, it implies that the Sobolev norm defined in \eqref{snrm} is equivalent to the usual Sobolev norm on any closed ball $B^{TM}(0,q)$, i.e. there exists a constant $C>0$ independent of the parameters $r,t$ such that
\[
\frac{1}{C(1+q)^{m+n}}\norm{\sigma}_{x,t,m} \leqslant \norm{\sigma}_m \leqslant C(1+q)^m\norm{\sigma}_{x,t,m}.
\]
\end{proof}

\subsection{Estimate on the resolvent of $\D_{1,t}^{x,2}$}
For the operator $\D_{1,t}$ it follows from the Lichnerowicz formula that
\begin{equation}\label{rsclnw}
\D_{1,t}^{x,2}=\Delta_{1,t}^x +\frac{t k^x(\sqrt{t}X)}{4}+t c_t(\F^x_1),
\end{equation}
where $k^x$ is the scalar curvature of $(\mathbb{R}^n,g^x)$ and $\F^x_1$ is the curvature of the connection $\nabla_1^{x,E_0}$.
From now on, set
\begin{equation}
\mathscr{M}^x_{1,t}=\Delta_{1,t}^x +\frac{t k^x(\sqrt{t}X)}{4}.
\end{equation}
In order to extract the precise contribution of the parameter $r$, we introduce an extra trick here, consisting of two steps. The first step is to establish the estimates for the operator $\M_{1,t}^x$ and its resolvent. Then the other step is to express the resolvent of $\D_{1,t}^{x,2}$ in terms that of $\M_{1,t}^x$, from which the contribution of the parameter $r$ emerges.

At first, for the operator $\M_{1,t}^x$, this estimate holds by its definition.
\begin{prop}\label{lemsp}
There exists constants $C_1,\,C_2,\,C_3>0$ such that for all $x\in M$, $t\in(0,t_0]$ and any $\sigma,\sigma'\in C^\infty_c(\mathbb{R}^n,S_{x}\otimes E_x)$, we have
\begin{align}
\label{M00}
&\inppro{\M^x_{1,t} \sigma, \sigma}_{t,0}\geqslant C_1 \norm{\sigma}^2_{t,1}-tC_2 \norm{\sigma}^2_{t,0}, \\
\label{M01}
&\abs{\inppro{\M^x_{1,t} \sigma, \sigma'}_{t,0}}\leqslant
C_3 \norm{\sigma}_{t,1}\norm{\sigma'}_{t,1} .
\end{align}
\end{prop}

\begin{proof}
It follows from the definition of the operator $\M^x_{1,t}$ that
\begin{equation}
\begin{split}
\inppro{\M^x_{1,t} \sigma, \sigma}_{t,0}=\inppro{\Delta^x_{1,t} \sigma, \sigma}_{t,0}+\inppro{\frac{t k^x(\sqrt{t}X)}{4}\sigma,\sigma}_{t,0}.
\end{split}
\end{equation}

The equation (\ref{M00}) and (\ref{M01}) now follows directly from it. 
\end{proof}

Then, as in \cite{BL} and \cite{DLM}, one has the following choice of  a contour $\Gamma$ and, accordingly, the estimate of the resolvent of $\M_{1,t}^x$.
\begin{prop}\label{nminvsm}
There exists a constant $A\in \mathbb{R}$ such that for $t\in(0,\frac{1}{2R}]$, if
\begin{equation}
\lambda \in \Gamma=\{\lambda \in \mathbb{C}: \mathrm{Re}(\lambda) \leqslant A \text{ and } \abs{\lambda-A}=3 \}\cup \{\lambda \in \mathbb{C}: \mathrm{Re}(\lambda) \geqslant A \text{ and } \abs{\mathrm{Im}(\lambda)}=3 \},
\end{equation} we have
\begin{align}
\label{00}\norm{(\lambda-\M_{1,t})^{-1}}_{x,t}^{0,0}&\leqslant \frac{1}{3},\\
\label{11}\norm{(\lambda-\M_{1,t})^{-1}}_{x,t}^{-1,1}&\leqslant
\frac{1}{3} (1+\abs{\lambda}^2).
\end{align}
\end{prop}
\begin{proof}
Choose $A=-3-\frac{C_2}{2R}$, for $C_2$ in \eqref{M00}, the inequality (\ref{00}) then follows directly from the Proposition \ref{lemsp}.

Taking a $\lambda_0\in \mathbb{R}$ such that $\lambda_0\leqslant -3-\frac{1}{2R}C_2$, then $(\lambda_0-\M_{1,t})^{-1}$ exists and satisfies
\begin{equation}
\norm{(\lambda_0 -\M_{1,t})^{-1}}^{-1,1}_{x,t} \leqslant \frac{1}{3}.
\end{equation}

Then, given any $\lambda \in \Gamma$, we have
\begin{equation}\label{lam0}
(\lambda-\M_{1,t})^{-1}=(\lambda_0-\M_{1,t})^{-1}-(\lambda-\lambda_0)(\lambda-\M_{1,t})^{-1}(\lambda_0-\M_{1,t})^{-1},
\end{equation}
it then follows that
\begin{equation}\label{n10}
\norm{(\lambda-\M_{1,t})^{-1}}^{-1,0}_t \leqslant \frac{1}{3} (1+\abs{\lambda-\lambda_0}),
\end{equation}
and, therefore, by applying \eqref{n10} again to \eqref{lam0},
\begin{equation}
\norm{(\lambda-\M_{1,t})^{-1}}^{-1,1}_t \leqslant \frac{1}{3} (1+\abs{\lambda}^2).
\end{equation}
\end{proof}

On the other hand, \eqref{rsclnw} tells that the operator $\D_{1,t}^{x,2}$ differs from $\M_{1,t}^x$ by
\begin{equation}
tc_t(\F_1^x)=\sum_{i \leqslant j}\F^x_{1,t,ij}(\varepsilon(e_i)-t\iota(e_i))(\varepsilon(e_j)-t\iota(e_j)),
\end{equation}
where $\F^s_{1,t,ij}: \mathbb{R}^n \rightarrow \ed{E_x}$ is determined by
\begin{equation}
\F^x_{1,t,ij}(X)=\F^x_1(\tilde{e}_i(t^{\frac{1}{2}}X),\tilde{e}_j(t^{\frac{1}{2}}X)), \qquad \forall X \in \mathbb{R}^n.
\end{equation}

Denote by $[T]^{(p)}$ for all $-n\leqslant p \leqslant n$ the part of an operator $T$ which maps $\mathbf{I}^0_{\ast,x}$ to $\mathbf{I}^0_{\ast+p,x}$ whenever applicable. Then the following two estimates make use of its property to establish the estimate on the resolvent of $\D_{1,t}^{x,2}$.

\begin{thm}
For all $k\geqslant 1$ and $t \in (0,\frac{1}{2R}]$
\begin{equation}
\norm{\left[\left( (\lambda-\M_{1,t}^x)^{-1} tc_t(\F_1^x)\right)^k\right]^{(p)}}^{0,0}_t \leqslant  (2R)^{\frac{p}{2}}2^{-k}.
\end{equation}
\end{thm}
\begin{proof}
Note that
\begin{equation}
\begin{split}
tc_t(\F_1^x)\varepsilon(F)&=\sum_{i<j} \F^x_{1,t,ij}\varepsilon(e_i)\varepsilon(e_j)
-t\F^x_{1,t,ij}(\iota(e_{i})\varepsilon(e_j)+\varepsilon(e_i)\iota(e_j))+t^2\F^x_{1,t,ij}\iota(e_i)\iota(e_j)\\
&:=\varepsilon(\F_{1,t}^x)+t \tau(\F_{1,t}^x)+t^2 \iota(\F_{1,t}^x).
\end{split}
\end{equation}
Then $\left[ (\lambda-\M_{1,t}^x)^{-1} tc_t(\F_1^x)\right]^k$ can be written as sum of terms as
\begin{equation}
\prod_{l=1}^k \left[(\lambda-\M_{1,t}^x)^{-1}\varepsilon(\F_1^x)\right]^{\alpha_l}\left[(\lambda-\M_{1,t}^x)^{-1}\iota(\F_1^x)\right]^{\beta_l}\left[(\lambda-\M_{1,t}^x)^{-1}\tau(\F_1^x)\right]^{\gamma_l},
\end{equation}
where $\alpha_l, \beta_l, \gamma_l$ are nonnegative integers such that $\alpha_l+\beta_l+ \gamma_l=1$.
And we set
\begin{equation}
\alpha=\sum_{l=1}^k \alpha_l, \qquad \beta=\sum_{l=1}^k \beta_l, \qquad \gamma=\sum_{l=1}^k \gamma_l,
\end{equation}
then $\alpha+\beta+\gamma=k$.


We prove by induction that, for all $k\geqslant 1$ and $\abs{p}\leqslant 2k$,
\begin{equation}
\norm{\left[\left((\lambda-\M_{1,t}^x)^{-1} tc_t(\F_1^x)\right)^k\right]^{(p)}}^{0,0}_{x,t}\leqslant R^{\frac{p}{2}}2^{\frac{p}{2}-k}.
\end{equation}
This is obviously true for $k=1$. Now assume this true for $k\in \mathbb{N}$.

\textbf{Case (i)}
If $k \leqslant \frac{n-3}{2}$, then
\begin{equation}
\begin{split}
&\left[\left((\lambda-\M_{1,t}^x)^{-1} tc_t(\F_1^x)\right)^{k+1}\right]^{(2k+2)}\\
=&\left[\left((\lambda-\M_{1,t}^x)^{-1} tc_t(\F_1^x)\right)^{k}\right]^{(2k)}(\lambda-\M_{1,t}^x)^{-1} \varepsilon(\F_1^x),
\end{split}
\end{equation}
and therefore
\begin{equation}
\norm{\left[\left((\lambda-\M_{1,t}^x)^{-1} tc_t(\F_1^x)\right)^{k+1}\right]^{(2k+2)}}^{0,0}_{x,t}\leqslant R^{k+1},
\end{equation}
as desired. It follows precisely the same way for $\left[\left((\lambda-\M_{1,t}^x)^{-1} tc_t(\F_1^x)\right)^{(k+1)}\right]^{-(2k+2)}$.

On the other hand, for $\abs{p}\leqslant 2k$,
\[
\begin{split}
&\left[\left((\lambda-\M_{1,t}^x)^{-1} tc_t(\F_1^x)\right)^{k+1}\right]^{(p)}\\
=&\left[\left((\lambda-\M_{1,t}^x)^{-1} tc_t(\F_1^x)\right)^{k}\right]^{(p-2)}(\lambda-\M_{1,t}^x)^{-1} \varepsilon(\F_1^x)\\
&+t\left[\left((\lambda-\M_{1,t}^x)^{-1} tc_t(\F_1^x)\right)^{k}\right]^{(p)}(\lambda-\M_{1,t}^x)^{-1} \tau(\F_1^x)\\
&+t^2 \left[\left((\lambda-\M_{1,t}^x)^{-1} tc_t(\F_1^x)\right)^{k}\right]^{(p+2)}(\lambda-\M_{1,t}^x)^{-1} \iota(\F_1^x).
\end{split}
\]
It then follows that
\begin{equation}\label{inductionkp}
\begin{split}
&\norm{\left[\left((\lambda-\M_{1,t}^x)^{-1} tc_t(\F_1^x)\right)^{k+1}\right]^{(p)}}^{0,0}_{x,t}\\
\leqslant &\frac{1}{3}\left(2^{\frac{p-2}{2}-k}R^{\frac{p-2}{2}}\cdot R+2^{\frac{p}{2}-k}t  R^\frac{p}{2} \cdot R + 2^{\frac{p+2}{2}-k} t^2 R^{\frac{p+2}{2}}\cdot R\right)\\
\leqslant & 2^{\frac{p}{2}-k-1} R^{\frac{p}{2}}.
\end{split}
\end{equation}
This covers the first case.

\textbf{Case (ii)} If $k\geqslant \frac{n-1}{2}$, then for all $\abs{p}\leqslant n-1$, the estimate follows the same way as \eqref{inductionkp} for $\left[\left((\lambda-\M_{1,t}^x)^{-1} tc_t(\F_1^x)\right)^{(k+1)}\right]^{-(2k+2)}$. And the desired estimates then follows.
\end{proof}

An explicit expression and a corresponding estimate of the resolvent $\left( \lambda -\D_{1,t}^{x,2}\right)^{-1}$ then follow from the estimate above.
\begin{thm}\label{drsv}
For $t\in (0,\frac{1}{2R}]$ and $\lambda\in \Gamma$, the resolvent of $\D_{1,t}^{x,2}$  is given by
\begin{equation}\label{invsrs}
(\lambda-\D_{1,t}^{x,2})^{-1}=(\lambda-\M_{1,t}^x)^{-1}\sum_{k=0}^\infty (tc_t(\F^x_1)(\lambda-\M_{1,t}^x)^{-1})^k.
\end{equation}
And there exists $C,C'>0$ such that, for all even integers $p\in [-n, n]$,
\begin{align}
&\norm{\left[(\lambda-\D_{1,t}^{x,2})^{-1}\right]^{(p)}}^{0,0}_{x,t} \leqslant C R^{\frac{p}{2}},\\
\label{rsv01}&\norm{\left[(\lambda-\D_{1,t}^{x,2})^{-1}\right]^{(p)}}^{-1,1}_{x,t} \leqslant C' R^{\frac{p}{2}}(1+\abs{\lambda}^2).
\end{align}
\end{thm}

\subsection{Regularizing properties of the resolvent of $\D_{1,t}^{x,2}$}
\begin{prop}\label{nblf}
Take $m\in \mathbb{N}^\ast$, there exists $C_m>0$, such that for all $t\in (0,\frac{1}{2R}]$ and $Q_1, Q_2,\dots, Q_m \in\{\nabla_{1,t,\partial_i}\}_{i=1}^n$ and $\sigma\in C_c^\infty (\mathbb{R}^n,S_{x}\otimes E_x)$ and $\sigma' \in C_c^\infty(\mathbb{R}^n, \Lambda^q(T^\ast_x M)\otimes E_x)$,
\begin{equation}\label{iR}
\begin{split}
&\abs{\inpro{([Q_1,[Q_2,\dots,[Q_m,c_t(\F_1^{x})]]\dots]) \sigma,\sigma'}_{x,t,0}}\\
\leqslant
&C_m \sum_{\abs{\alpha}\leqslant m}\norm{X^\alpha\sigma'}_{x,t,0}\left(R\norm{\left[\sigma\right]^{(q-2)}}_{x,t,0}+\norm{\left[\sigma\right]^{(q)}}_{x,t,0}+t\norm{\left[\sigma\right]^{(q+2)}}_{x,t,0}\right).
\end{split}
\end{equation}
\end{prop}

\begin{proof}
For $[\nabla_{1,t,\partial_i},tc_t(\F^x_1)]$, we have
\begin{equation}
\begin{split}
t[\nabla_{1,t,\partial_i},c_t(\F^x_1)]&=t^{\frac{3}{2}}\delta_{t}[\nabla_{1,\partial_i},c(\F_1^x)]\delta_{t}^{-1}\\
&=t^{\frac{3}{2}}\delta_{t}[\nabla_{1,\partial_i},\sum_{j<k}\F^x_{1,jk}c(e_j)c(e_k)]\delta_{t}^{-1}\\
&=t^{\frac{3}{2}}\delta_{t}[\nabla_{1,e_i},\sum_{j<k}\F^x_{1,jk}]\delta_{t}^{-1}c_t(e_j)c_t(e_k).
\end{split}
\end{equation}
In particular, for the commutator $[\nabla_{1,\partial_i},\sum_{j<k}\F^x_{1,jk}]$, it follows from \eqref{cnctform} that
\begin{equation}
\g{t}{[\nabla_{1,\partial_i},\sum_{j<k}\F^x_{1,jk}]} (X)=\sum_{j<k}\partial_i\F^x_{1,jk}(t^{\frac{1}{2}}X)+\sum_{j<k}[\A^x_1(\partial_i),\F^x_{1,jk}](t^{\frac{1}{2}}X).
\end{equation}
Furthermore, note that
\begin{equation}
\abs{(\partial_i\F^x_{1,jk})\vert_{t^{\frac{1}{2}}X}}\leqslant C_0 R,
\end{equation}
and since, by \eqref{cnctform}, $\abs{\A^x_{1,t}(\partial_i)\vert_{X}}\leqslant C_1 \abs{X}$ for some $C_1>0$ holds pointwisely for all $X\in \mathbb{R}^n$, it then follows that
\begin{equation}
t^{\frac{1}{2}}\abs{[\A^x_{1}(\partial_i),\F_{1,jk}^x]\vert_{t^{\frac{1}{2}}X}} \leqslant C'_0 R \abs{X},
\end{equation}
for some $C_0,C_0'>0$. Recall that, by Proposition \ref{unbdactions}, the operators $t^{\frac{1}{2}}c_t(e_j)$'s and $t^{\frac{1}{2}}\abs{X}c_t(e_j)$'s are uniformly bounded. The desired estimate then follows for $m=1$. Then, for $m>1$, the estimate hold by repeating the same procedure.
\end{proof}

\begin{prop}\label{nblm}
Take $m\in \mathbb{N}^\ast$, there exists $C_m>0$, such that for $t\in (0,\frac{1}{2R}]$ and $Q_1, Q_2,\dots, Q_m \in\{\nabla_{1,t,\partial_i},X_i \}_{i=1}^n$ and $\sigma,\sigma' \in C_c^\infty (\mathbb{R}^n,S_{x}\otimes E_x)$,
\begin{equation}\label{cmtqm}
\abs{\inpro{([Q_1,[Q_2,\dots,[Q_m,\M_{1,t}^{x}]]\dots]) \sigma,\sigma'}_{t,0}}\leqslant
C_m \sum_{\abs{\alpha}\leqslant m}\norm{X^\alpha \sigma}_{t,1}\norm{\sigma'}_{t,1}.
\end{equation}
\end{prop}

\begin{proof}
Under the frame $\{\partial_i=\frac{\partial}{\partial X^i}\}_{i=1}^n$, we write the rescaled connection Laplacian as
\begin{equation}
\Delta^x_{1,t}=-g^{ij}(t^{\frac{1}{2}}X)\left(\nabla_{1,t,\partial_i}\nabla_{1,t,\partial_j}-t^{\frac{1}{2}}\Gamma_{ij}^k(t^{\frac{1}{2}}X)\nabla_{1,t,\partial_k}\right),
\end{equation}
where the Christoffel symbols $\Gamma_{ij}^k$ is defined via $\nabla^{x,TM}_{\partial_i} \partial_j=\sum_k \Gamma_{ij}^k \partial_k$.

Note that $[\nabla_{1,t,\partial_i},X^j]=\delta_{ij}$, the commutator $[X^j, \M_{1,t}^x]$ verifies \eqref{cmtqm}. And note that, according to\cite[Proposition 3.43]{BGV},
\begin{equation}
\begin{split}
[\nabla^x_{1,\partial_i},\nabla^x_{1,\partial_j}]=&\R^{x, S_0\otimes E_0}(\partial_i, \partial_j)\\
=& \R^{x,S_0}(\partial_i, \partial_j)+\F^x(\partial_i, \partial_j)\\
=& \frac{1}{4}\inpro{\R^{x,TM}(\partial_i, \partial_j)\tilde{e}_k,\tilde{e}_l}c(e_k) c(e_l)+\F^x(\partial_i, \partial_j),
\end{split}
\end{equation}
then, by applying the rescaling, one has
\begin{equation}
\begin{split}
[\nabla^x_{1,t,\partial_i},\nabla^x_{1,t,\partial_j}]=& t\g{t}{[\nabla^x_{1,\partial_i},\nabla^x_{1,\partial_j}]}\\
&=\frac{t}{4}\inpro{\R^{x,TM}(\partial_i, \partial_j)\tilde{e}_k,\tilde{e}_l}(t^{\frac{1}{2}}X)c_t(e_k) c_t(e_l)+t\F^x(\partial_i, \partial_j)(t^{\frac{1}{2}}X).
\end{split}
\end{equation}

It then follows that, the operator$[\nabla_{1,t,\partial_i}, \M_{1,t}]$ has the type as
\begin{equation}\label{strcqm}
\sum_{ij}a_{ij}(t,t^{\frac{1}{2}}X)\nabla_{1,t,\partial_i}\nabla_{1,t,\partial_j}+\sum_i b_{i}(t,t^{\frac{1}{2}}X)\nabla_{1,t,\partial_i}+c(t,t^{\frac{1}{2}}X),
\end{equation}
where $a_{ij}(t,X)$, $b_{i}(t,X)$, and $c(t,X)$ are polynomials in the variable $t^{\frac{1}{2}}$ with coefficients taken from functions in $X$ whose derivatives are uniformly bounded for $X\in \mathbb{R}^n$.

Denoting by $\nabla_{1,t,\partial_i}^{x,\ast}$ the adjoint of $\nabla_{1,t,\partial_i}^{x}$ with respect to $\inppro{\cdot, \cdot}_{x,t,0}$, and by $\left(\nabla_{1,t,\partial_i}^{x,\ast}\right)^{[p]}$ its restriction on $I_{p,x}^0$, one then has
\begin{equation}\label{adjt}
\left(\nabla_{1,t,\partial_i}^{x,\ast}\right)^{[p]}=-\nabla_{1,t,\partial_i}^{x}-t^{\frac{1}{2}}(k_p^{-1}\partial_i k_p)(t^{\frac{1}{2}}X),
\end{equation}
where the function $k$ is defined by $k_p(X)=\mathrm{det}(g^x)(1+f_\delta(t^{\frac{1}{2}}\abs{X})\abs{X})^{2(n-p)}$ and the derivatives of the $k_p$'s are uniformly bounded in $X\in \mathbb{R}^n$. It the follows that \eqref{cmtqm} holds form $m=1$.

For $m>1$, one sees by induction that $[Q_1,[Q_2,\dots,[Q_m,\M_{1,t}^{x}]\dots]$ has the same structure as \eqref{strcqm} and the desired estimate follows from \eqref{adjt}.
\end{proof}

Combining the these estimates together, we have the following estimate for the resolvent.
\begin{thm}
For any $t\in (0,\frac{1}{2R}]$, $\lambda \in \Gamma$ and $m\in \mathbb{N}^\ast$, $(\lambda-\D_{0,t}^{x,2})^{-1}$ maps $\mathbf{I}^m_x$ into $\mathbf{I}^{m+1}_x$. Moreover, for any multi-index $\alpha\in \mathbb{Z}^n$, there exist $N,N'\in \mathbb{N}$ and $C_{\alpha,m}>0$, such that, for any $\sigma\in C_c^\infty (\mathbb{R}^n, \Lambda^q(T^\ast_x M) \otimes E_x)$ and $\lambda \in \Gamma$,
\begin{equation}\label{rsvmm1}
\norm{\left[X^\alpha (\lambda-\D_{1,t}^{x,2})^{-1} \sigma \right]^{(p+q)}}_{t,m+1} \leqslant C_{\alpha,m} R^{\frac{p}{2}}(1+\abs{\lambda}^2)^N \sum_{\abs{\alpha'}\leqslant N'}\norm{X^{\alpha'}\sigma}_{t,m},
\end{equation}
if $\abs{p+q}\leqslant n-1$, and vanishes otherwise.
\end{thm}

\begin{proof}
Taking $Q_1,Q_2,\dots , Q_m\in \{\nabla_{1,t,\partial_i}\}_{i=1}^n$, and $Q_{m+1},Q_{m+2},\dots, Q_{m+{\abs{\alpha}}} \in \{X^i\}_{i=1}^n$, the operator $Q_1 Q_2 \dots Q_{m+\abs{\alpha}} (\lambda -\D_{1,t}^{x,2})^{-1}$ can be written as a linear combination of the operators of the type
\[
[Q_1,[Q_2,\dots[Q_{m'},(\lambda-\D_{1,t}^{x,2})^{-1}]]\dots]Q_{m'+1}\dots Q_{m+\abs{\alpha}},
\]
for $m'\leqslant m$.

Denote by $\mathscr{R}_t$ the collection of the operators of the form $[Q_{j_1},[Q_{j_2},\dots[Q_{j_l},\D_{1,t}^{x,2}]]\dots]$. Clearly, any commutator $[Q_1,[Q_2,\dots[Q_{m'},(\lambda-\D_{1,t}^{x,2})^{-1}]\dots]]$ is a linear combination of operators of the form
\begin{equation}
(\lambda-\D_{1,t}^{x,2})^{-1}R_1(\lambda-\D_{1,t}^{x,2})^{-1}R_2\dots R_{m'}(\lambda-\D_{1,t}^{x,2})^{-1}
\end{equation}
with $R_1,R_2,\dots, R_{m'}\in \mathscr{R}_t$.
By the proposition above provides a uniform estimate for the norm $\norm{R_j}_{x,t}^{1,0}$. And by \eqref{rsv01} we find that there exists a constant $C>0$ and $M,N\in \mathbb{N}$ such that, for all $\sigma\in C_c^\infty (\mathbb{R}^n, S_x \otimes E_x)$,
\begin{equation}
\begin{split}
&\norm{\left[(\lambda-\D_{1,t}^{x,2})^{-1}R_1(\lambda-\D_{1,t}^{x,2})^{-1}R_2\dots R_{m'}(\lambda-\D_{1,t}^{x,2})^{-1}\sigma\right]^{(p+q)}}_{x,t,1}\\
<&C_{m'}R^{\frac{p}{2}}(1+\abs{\lambda}^2)^N\sum_{\abs{\alpha}\leqslant M}\norm{X^\alpha \sigma}_{x,t,0}.
\end{split}
\end{equation}
The theorem then follows.
\end{proof}

\subsection{Uniform estimate on the kernel $\P^2_{x,t}$}
As previously mentioned, the problem has been converted from estimating heat kernel $P(t;x,x)$ for small time $t$, which appears to be singular at $t=0$, to a uniform heat kernel estimate estimate for a smooth family of elliptic operators which does not have singularity.  And the discussion from the previous sections shows that it now suffices to provide a uniform estimate of the kernel $\P^{2}_{x,t}(u;x,y)$.

The setup above allowed us to use the method in \cite{BL} and \cite{DLM}. In this part we are going to prove
\begin{thm}\label{knlest}
There exists a $C>0$ and $\alpha>0$ such that for any $t\in(0, \frac{1}{2R}]$ and $u>0$,
\begin{equation}\label{K1}
\sup_{X,Y \in B_{\epsilon}}\abs{\P^2_{x,t}(u;X,Y)} \leqslant Ct^{\frac{1}{2}}R^{\frac{n-1}{2}} e^{\alpha u}
\end{equation}
and similarly for the rescaled heat kernel, 
\begin{equation}
\sup_{X,Y \in \mathbb B_{\epsilon}}\abs{\P^0_{x,t}(u;X,Y)} \leqslant CR^{\frac{n-1}{2}} e^{\alpha u}.
\end{equation}
\end{thm}

\begin{rmk}
The constant $C$ in Theorem \ref{knlest} is independent of the parameter $r$ and the variable $u$, and this is the reason why it is said to be uniform.
\end{rmk}

The desired estimate of $\eta$-invariant then follows immediately from the estimate \eqref{K1} and \eqref{treq}, by taking $u=1$ in \eqref{trvsrscl}.
\begin{prop}
There exists a constant $C>0$ that is independent of the parameter $r$ such that
\begin{equation}
\int_0^{t_0} t^{-\frac{1}{2}}\Tr{D_1\exp{(-tD_1^2)}}\d t\leqslant CR^{\frac{n-2}{2}}.
\end{equation}
\end{prop}
This is the estimate of the second term on the right-hand side of (\ref{etasp}), i.e. the ``small-time" part of that integral in \eqref{eta}.

Recall As in \cite{BL} and \cite{DLM}, it follows from Theorem \ref{drsv} that one can rewrite the heat operator in the following way.

\begin{prop}
For any $k\in \mathbb{N}^\ast$, one has
\begin{equation}\label{cauchy}
\exp{(-u\D_{1,t}^{x,2})}=\frac{(-1)^{k-1}(k-1)!}{2\pi \sqrt{-1}u^{k-1}}\int_\Gamma e^{-u\lambda}(\lambda-\D_{1,t}^{x,2})^{-k}\d\lambda.
\end{equation}
\end{prop}

Furthermore, by applying \eqref{cauchy} to \eqref{opltx}, it follows that, given any $k\in \mathbb{N}^\ast$,
\begin{equation}
\begin{split}
&[\exp{(-t \D_{1,t}^{x,2})},\mathcal{R}^c_t]\\
=&\frac{(-1)^{k-1}(k-1)!}{2\pi \sqrt{-1}u^{k-1}}\int_\Gamma e^{-u\lambda}[(\lambda-\D_{1,t}^{x,2})^{-k},\mathcal{R}^c_t]\d\lambda,
\end{split}
\end{equation}
where
\begin{equation}
\begin{split}
[(\lambda-\D_{1,t}^{x,2})^{-k},\mathcal{R}^c_t]&=\sum_{i=1}^{k}(\lambda-\D_{1,t}^{x,2})^{-i+1}[(\lambda-\D_{1,t}^{x,2})^{-1},\mathcal{R}^c](\lambda-\D_{1,t}^{x,2})^{-k+i}\\
&=\sum_{i=1}^{k}(\lambda-\D_{1,t}^{x,2})^{-i}[\D_{1,t}^{x,2},\mathcal{R}^c_t](\lambda-\D_{1,t}^{x,2})^{-k+i-1}.
\end{split}
\end{equation}

It then follows that
\begin{equation}\label{rsvntl}
\begin{split}
&\mathscr{L}_t^x(u)=u^{\frac{1}{2}}\D_{1,t}^x \exp{(-t \D_{1,t}^{x,2})}-u^{-\frac{1}{2}}[\exp{(-t \D_{1,t}^{x,2})},\mathcal{R}^c_t]\\
=&\frac{(-1)^{k}k!}{2\pi \sqrt{-1}u^{k-\frac{1}{2}}}\int_\Gamma e^{-u\lambda}\D_{1,t}^x(\lambda-\D_{1,t}^{x,2})^{-k-1}\d\lambda\\
&+\frac{(-1)^{k-1}(k-1)!}{2\pi \sqrt{-1}u^{k-\frac{1}{2}}}\int_\Gamma e^{-u\lambda}[(\lambda-\D_{1,t}^{x,2})^{-k},\mathcal{R}^c_t]\d\lambda\\
=&\frac{(-1)^{k}(k-1)!}{2\pi \sqrt{-1}u^{k-\frac{1}{2}}}\int_\Gamma e^{-u\lambda}\sum_{i=1}^{k}(\lambda-\D_{1,t}^{x,2})^{-i}([\D_{1,t}^{x,2},\mathcal{R}^c_t]+\D_{1,t}^x)(\lambda-\D_{1,t}^{x,2})^{-k+i-1}\d\lambda.
\end{split}
\end{equation}

Set the operator
\begin{equation}\label{l1}
\L_{1,t}^x=\frac{1}{2}[\D_{1,t}^{x,2},\mathcal{R}^c_t]+\D_{1,t}^x,
\end{equation}
for which we have the following estimate.
\begin{prop}
For any $m\in \mathbb{N}$, there exists a constant $C_m>0$ such that, for all $\sigma\in C^\infty_c(\mathbb{R}^n, S_x \otimes E_x)$,
\begin{equation}\label{l1t}
\norm{\L_{1,t}^x\sigma}_{x,t,m}\leqslant  C_m t^{\frac{1}{2}}\sum_{\abs{\alpha}\leqslant m+1}\norm{X^\alpha\sigma}_{x,t,m+1}.
\end{equation}
\end{prop}
\begin{proof}
Denoting by $\mathcal{R}$ the radial vector field, we start with the operator
\begin{equation}
[\D_{1}^{x,2},\mathcal{R}^c]=-2\sum_{i,j}c(\nabla^{x,TM}_{\tilde{e}_j} \mathcal{R}) \nabla_{1,\tilde{e}_j}^x+c(\Delta^{x,TM} \mathcal{R}),
\end{equation}
on which the Getzler's reslating will later be applied.

According to \cite[Proposition 1.27]{BGV}, one has that $\mathcal{R}=\nabla\abs{X}^2$. And therefore,
\begin{equation}
\begin{split}
c(\Delta^{x,TM} \mathcal{R})=&c(\Delta^{x,TM}\nabla\abs{X}^2)\\
=&c(\nabla \Delta^{x,TM}\abs{X}^2)-c(\mathrm{Ric}(\mathcal{R}))\\
=&-\sum_{i=1}^n \tilde{e}_i \mathcal{R}\log{ (\mathrm{det}(g^{x,TM}))}c(e_i)-\mathrm{Ric}(\mathcal{R},\tilde{e}_i)c(e_i)\\
=&-\sum_{i=1}^n \tilde{e}_i \mathcal{R}\tr{\log{(g^{x,TM})}}c(e_i)-\mathrm{Ric}(\mathcal{R},\tilde{e}_i)c(e_i),
\end{split}
\end{equation}
where, by $\mathrm{Ric}(\mathcal{R})$, we denote the tangent vector field such that for all 
vector fields $\mathcal{V}$,
\begin{equation}
\inpro{\mathrm{Ric}(\mathcal{R}),\mathcal{V}}=\mathrm{Ric}(\mathcal{R},\mathcal{V}).
\end{equation}
On the other hand, according to \cite[Proposition 1.27]{BGV}, the radial vector field $\mathcal{R}=\sum_{i=1}^n X^i \tilde{e}_i$ can also be expressed as $\mathcal{R}=\sum_{i=1}^n X^i \frac{\partial}{\partial X^i}$. It now follows for the first term that
\begin{equation}
\begin{split}
\sum_{i,j}c(\nabla^{T_x M}_{\tilde{e}_j} \mathcal{R}) \nabla_{1,\tilde{e}_j}^x=&\sum_{i,j}\d X^{i}(\tilde{e}_j)c(\partial_i) \nabla_{1,\tilde{e}_j}^x+X^i c(\nabla_{\tilde{e}_j}^{T_x M} \partial_i ) \nabla_{1,\tilde{e}_j}^x\\
=&\D_{1}^x+\sum_{i,j,k}X^i \Theta^k_j  c(\nabla^{T_x M}_{\partial_k} \partial_i )\nabla_{1,\tilde{e}_j}^x\\
=&\D_{1}^x+\sum_{i,j,k}X^i \Theta^k_j  c(\nabla^{T_x M}_{\partial_i} \partial_k )\nabla_{1,\tilde{e}_j}^x\\
=&\D_{1}^x+\sum_{j,k} \Theta^k_j  c(\nabla^{T_x M}_{\mathcal{R}} \partial_k )\nabla_{1,\tilde{e}_j}^x\\
=&\D_{1}^x+\sum_{j,l,p}  (\mathcal{R}\Theta^p_j) (\Theta^{-1})^l_p c(e_l)\nabla_{1,\tilde{e}_j}^x,
\end{split}
\end{equation}
where $\Theta^k_j=\d X^k(\tilde{e}_j)$.

Therefore, one can write
\begin{equation}
\begin{split}
&\D_1^x+[\D_1^{x,2},\mathcal{R}^c]\\
=&-\sum_{i,j,k}  (\mathcal{R}\Theta^j_i) (\Theta^{-1})^k_j c(e_k)\nabla_{1,\tilde{e}_j}^x- \frac{1}{2}\sum_{i=1}^n\left(\tilde{e}_i \mathcal{R}\tr{\log{(g^{x,TM})}}c(e_i)-\mathrm{Ric}(\mathcal{R},\tilde{e}_i)c(e_i)\right).
\end{split}
\end{equation}

According to \cite[Proposition 1.28]{BGV}, one has the Taylor expansion of $\Theta$ and $g^{x,TM}$,
\begin{equation}
\begin{split}
\Theta_i^j=\delta_{ij}-\frac{1}{6}\sum_{k,l}X^k X^l R_{jkil}(x)+O(\abs{X}^3),\\
g_{ij}=\delta_{ij}-\frac{1}{3}\sum_{k,l}X^k X^l R_{jkil}(x)+O(\abs{X}^3).
\end{split}
\end{equation}

Since
\[
\L_{1,t}^x=\D_{1,t}^x+[\D_{1,t}^{x,2},\mathcal{R}_t^c]=t^{\frac{1}{2}}\g{t}{\D_{1}^x+[\D_{1}^{x,2},\mathcal{R}^c]},
\]
it then follows that, there exists some $C_0>0$ such that,
\begin{equation}
\begin{split}
&\norm{\sum_{i,j,k}  (t^{\frac{1}{2}}\delta_t\mathcal{R}\Theta^j_i) (\Theta^{-1})^k_j c(e_k)\nabla_{1,\tilde{e}_j}^x\delta_{t}^{-1}\sigma}_{x,t,0}\leqslant C_0 t \sum_{\abs{\alpha}\leqslant1}\norm{X^{\alpha}\sigma}_{x,t,1},\\
&\norm{\sum_{i,j,k}  (t^{\frac{1}{2}}\delta_t\frac{1}{2}\sum_{i=1}^n\left(\tilde{e}_i \mathcal{R}\tr{\log{(g^{x,TM})}}c(e_i)-\mathrm{Ric}(\mathcal{R},\tilde{e}_i)c(e_i)\right)\delta_{t}^{-1}\sigma}_{x,t,0}\leqslant C_0 t \norm{\sigma}_{x,t,0},\
\end{split}
\end{equation}
for all $\sigma\in C_c^\infty (\mathbb{R}^n, S_x\otimes E_x)$, and, therefore,
\begin{equation}
\norm{\L_{1,t}^x \sigma}_{t,0} \leqslant C_0 t^{\frac{1}{2}} \sum_{i=1}^n\norm{X^i\sigma}_{t,1}.
\end{equation}

Furthermore, from the same procedure as in the proof of Proposition \ref{nblm} and \ref{nblf}, it follows that there exists $C_m>0$ for each $m\in\mathbb{N}$ such that
\begin{equation}
\norm{\L_{1,t}^x \sigma}_{t,m} \leqslant C_m t^{\frac{1}{2}} \sum_{\abs{\alpha}<m+1}\norm{X^{\alpha}\sigma}_{t,m+1},
\end{equation}
for all $\sigma\in C_c^\infty (\mathbb{R}^n, S_x\otimes E_x)$.
\end{proof}

\begin{proof}[Proof of Theorem \ref{knlest}]
Denote by $C_{\epsilon}^{\infty}(\mathbb{R}^n, S_x\otimes E_x)$ the set of smooth sections of the bundle $S_0 \otimes E_0$ which are supported within the open ball $B_{\epsilon} (0) \subset \mathbb{R}^n$. Take an arbitrary $\sigma\in C_\epsilon^\infty (\mathbb{R}^n, S_x \otimes E_x)$. Then for any $Q=Q_1Q_2\dots Q_k,\tilde{Q}=Q_{k+1}\dots Q_{k+l}$, with $Q_j$'s taken from $\{\nabla_{1,t,\partial_i} \}_{i=1}^n$, use the expression
\begin{equation}
Q\L_t^x \tilde{Q}= \frac{(-1)^{k+l}(k+l-1)!}{2\pi \sqrt{-1}u^{k+l-\frac{1}{2}}}\int_\Gamma e^{-u\lambda}\sum_{i=1}^{k+l} \left(Q(\lambda-\D_{1,t}^{x,2})^{-i}\L_{1,t}^x(\lambda-\D_{1,t}^{x,2})^{-k-l+i-1}\tilde{Q} \right) \d\lambda,
\end{equation}
where, according to \eqref{rsvmm1} and \eqref{l1t},
\begin{equation}
\norm{Q(\lambda-\D_{1,t}^{x,2})^{-i}\L_{1,t}^x(\lambda-\D_{1,t}^{x,2})^{-2m+i-1}\tilde{Q} \sigma }_{x,t,0}\leqslant C(m) R^{\frac{n-1}{2}}t^{\frac{1}{2}}(1+\abs{\lambda}^2)^N \norm{\sigma}_{x,t,0},
\end{equation}
for some $N>0$ and $C(m)>0$ both independent of $i$.
This means that there exist $\tilde{C}(m), \alpha>0$ such that, for all $u>0$,
\begin{equation}
\norm{\int_{\mathbb{R}^n}Q^{(1)} \tilde{Q}^{\ast,(2)}\P_{x,t}^2(u;X,Y) \sigma(Y) \d Y}_{x,t,0}\leqslant \tilde{C}(m)R^{\frac{n-1}{2}}t^{\frac{1}{2}}e^{\alpha u} \norm{\sigma}_{x,t,0}.
\end{equation}
It then follows from the Sobolev inequality that
\begin{equation}
\sup_{X\in B_{\delta/2}(0)} \abs{\int_{\mathbb{R}^n} \tilde{Q}^{\ast,(2)}\P_{x,t}^2(u;X,Y) \sigma(Y) \d Y}\leqslant C_1 R^{\frac{n-1}{2}}t^{\frac{1}{2}} e^{\alpha u}\norm{\sigma}_{x,t,0}.
\end{equation}
Since the section $\sigma$ is taken arbitrarily, it follows that, for all fixed $X\in B_{\epsilon}(0)$,
\begin{equation}
\norm{\tilde{Q}^{\ast,(2)}\P_{x,t}^2(u;X,Y)}_{x,t,0}\leqslant C_1 R^{\frac{n-1}{2}}t^{\frac{1}{2}}e^{\alpha u}.
\end{equation}
Applying Sobolev inequality again gives
\begin{equation}
\sup_{\abs{X},\abs{Y}<\epsilon} \abs{\P_{x,t}^2(u;X,Y)}	\leqslant CR^{\frac{n-1}{2}}t^{\frac{1}{2}} e^{\alpha u}.
\end{equation}
And similarly, it can be shown that
\begin{equation}
\sup_{\abs{X},\abs{Y}<\epsilon} \abs{\P_{x,t}^0(u;X,Y)}	\leqslant CR^{\frac{n-1}{2}} e^{\alpha u}.
\end{equation}

Thus we get the desired estimate in Theorem \ref{knlest}.
\end{proof}

\subsection{The Estimate of $\eta$-invariant}
Now the proof of Theorem \ref{main2} follows from the uniform estimate of $\P_{x,t}^2(u;X,Y)$, the equation (\ref{etasp}) and the Proposition \ref{etast}.

\mainb*

\begin{proof}
Note that 

\[
\begin{split}
\bar{\eta}{(D_1)} &=\frac{1}{2}(\dim \Ker{D_1}+\eta(D_1))\\
&=\frac{1}{2}(\dim \Ker{D_1}+\int_{0}^\infty t^{-\frac{1}{2}} \Tr{D_1 \exp{(-tD_1^2)}})\d t\\
&=\frac{1}{2}(\dim \Ker{D_1}+\int_{t_0}^\infty t^{-\frac{1}{2}} \Tr{D_1 \exp{(-tD_1^2)}}\d t+\int_{0}^{t_0} t^{-\frac{1}{2}} \Tr{D_1 \exp{(-tD_1^2)}}\d t).
\end{split}
\]

From the estimate achieved above, it follows that there exists constants $C_1,C_2 >0$ such that
\[
\begin{split}
\abs{\dim \Ker{D_1}+\int_{t_0}^\infty t^{-\frac{1}{2}} \Tr{D_1 \exp{(-tD_1^2)}}\d t}&\leqslant C_1 R^{\frac{n}{2}},\\
\abs{\int_0^{t_0} t^{-\frac{1}{2}} \Tr{D_1 \exp{(-tD_1^2)}}\d t}\leqslant C_2 R^{\frac{n-2}{2}}.
\end{split}
\]

Combining the estimates above provides the desired estimate for the reduced $\eta$-invariant.
\end{proof}

In conclusion, suming up all the results above,  it eventually provides our estimate of the asymptotic spectral flow.

\maina*

\subsection{Proof of Lemma \ref{limofker}}
It has been claimed in the previous section that the limit $\lim_{t\rightarrow 0} c_t(\hat{a})\delta_t \K_s(1;X)$ of the rescaled kernel equals to the corresponding kernel $\varepsilon(\hat{a})K_0(1;0)$ of the limiting operator. Denote by $\P^0_{x_0,s,t}(u;X,Y)$ the kernel associated to $\exp{(-u\D_{s,t}^{x_0,2})}$, and by $\K_{s,0}(u;X,Y)$ that of the operator $\exp{(-u\L_{s,0}^{x_0})}$, where $\L_{s,0}^{x_0}=\lim_{t\rightarrow 0} \D_{s,t}^{x_0,2}$ given by \eqref{limrsdf}. With the technique in this section, here we sketch a proof for it by showing the following estimate.

\begin{prop}\label{esterr}
There exists a constant $C>0$ such that for all $t\in[0,\frac{1}{2R}]$, $u\geqslant 0$
\begin{equation}
\sup_{X,Y\in B_{\epsilon}(0)}\abs{\P^0_{x_0,s,t}(1;X,Y)-\K_{s,0}(1;X,Y)}\leqslant Ct^{\frac{1}{2}} R^{\frac{n+1}{2}}.
\end{equation}
\end{prop}

For any section $\sigma\in C_c^\infty(\mathbb{R}^n;S_{x_0}\otimes E_{x_0})$, we use the following norms 
\begin{equation}
\begin{split}
&\norm{\sigma}_{x_0,0}^2=\lim_{t\rightarrow 0} \norm{\sigma}_{x_0,t,0}=\int_{\mathbb{R}^n} \abs{\sigma(X)}^2 \left( 1+\abs{X} \right)^{2(n-p)} \d X,\\
&\norm{\sigma}_{x_0,m}=\lim_{t\rightarrow 0} \norm{\sigma}_{x,t,m}=\sum_{l=0}^{m}\sum_{i_1,\dots,i_l=1}^{n}\norm{\nabla_{0,\partial_{i_l}},\dots,\nabla_{0,\partial_{i_l}}^x \sigma}^2_{x_0,0},
\end{split}
\end{equation}
where, as in \eqref{limlaplace}, 
\[
\nabla_{0,\partial_{i}}=\partial_i+\frac{1}{4}\sum_{j;k<l}X^jR_{ijkl}^{x_0}\varepsilon(e^k)\varepsilon(e^l).
\]

Under this norm, we have similar estimates for the resolvent of $\D_{s,t}^{x_0,2}$ and $\L_{s,0}^{x_0}$ as following by exactly the same procedure for \eqref{rsvmm1}.
\begin{lem}
For all $m\in \mathbb{N}$ and multi-index $\alpha$, there exist constants $C_{m,\alpha}>0$ and $N,M\in \mathbb{N}$ such that for all $s\in [0,1]$, $t\in (0,\frac{1}{2R}]$, and any $\sigma\in C_c^\infty(\mathbb{R}^n, \Lambda^q (T_{x_0}M) \otimes E_{x_0})$, 
\begin{equation}\label{limrsv}
\begin{split}
&\norm{\left[X_\alpha(\lambda-\D_{s,t}^{x_0,2})^{-1}\sigma\right]^{(p+q)}}_{m+1}\leqslant C_{\alpha,m}R^{\frac{p}{2}}(1+\abs{\lambda})^{N}\sum_{\alpha'\leqslant M} \norm{X^{\alpha'}\sigma}_{x_0,0},\\
&\norm{\left[X_\alpha(\lambda-\L_{s,0}^{x_0})^{-1}\sigma\right]^{(p+q)}}_{m+1}\leqslant C_{\alpha,m}R^{\frac{p}{2}}(1+\abs{\lambda})^N \sum_{\abs{\alpha'}\leqslant M}\norm{X^{\alpha'}\sigma}_{x_0,0}.
\end{split}
\end{equation}
\end{lem}

The estimate is largely the same as that of Theorem \ref{knlest}, the only difference is the operator $\D_{s,t}^{x_0,2}-\L_{s,0}^{x_0}$ that will be involved.
\begin{prop}
There exists a constant $C>0$ such that for all $t\in[0,\frac{1}{2R}]$, $\sigma,\sigma' \in C_c^\infty(\mathbb{R}^n; S_{x_0}\otimes E_{x_0})$,
\begin{equation}
\inppro{\left(\D_{s,t}^{x_0,2}-\L_{s,0}^{x_0} \right)\sigma,\sigma'}_{x_0,0}\leqslant Ct^{\frac{1}{2}}R\sum_{\abs{\alpha}<4}\norm{X^\alpha \sigma}_{x_0,1}\norm{\sigma'}_{x_0,1}.
\end{equation}
\end{prop}
\begin{proof}
By direct computation, we have
\begin{equation}
\D_{s,t}^{x_0,2}-\L_{s,0}^{x_0}=(\Delta_{s,t}-\L)+\frac{tk^x_t}{4}+t(c_t(F_s)-\varepsilon(F_s)),
\end{equation}
where $\L$ is the generalized harmonic oscillator.
It is easy to see, for any $\sigma \in C_c^\infty(\mathbb{R}^n; S_{x_0}\otimes E_{x_0})$, that 
\begin{equation}
\norm{\frac{tk^x_t}{4}\sigma}_{x_0,0}\leqslant C_1 t \norm{\sigma}_{x_0,0}.
\end{equation}
And note that 
\[(tc_t(\F_s^{x_0})-\varepsilon(\F_s^{x_0}))=\sum_{i<j} \F_{ij}^{x_0}(t^{\frac{1}{2}}\iota(e_i)\varepsilon(e^j)+t^{\frac{1}{2}}\varepsilon(e^i)\iota(e_j)+t\iota(e_i)\iota(e_j)),
\]
it then follows directly that
\begin{equation}
\norm{(tc_t(\F_s^{x_0})-\varepsilon(\F_s^{x_0}))\sigma}_{x_0,0}\leqslant C_2t^{\frac{1}{2}}R \norm{\sigma}_{x_0,0}.
\end{equation}
For the operator $\Delta_{s,t}-\L$, note that
\begin{equation}
\nabla_{s,t,\tilde{e}_i}-\nabla_{0,\partial_i}=(\Theta^j_i -\delta_{i}^j)\vert_{\sqrt{t}X}\partial_j+(\omega_t(\tilde{e}_i)-\Omega_{ij}X^j)+A_{s,t}(\tilde{e}_i),
\end{equation}
for each of the operators on the right-hand side, there exist $C_3',C_3'',C_3'''>0$ such that
\begin{equation}
\begin{split}
&\norm{(\Theta^j_i -\delta_{i}^j)\vert_{\sqrt{t}X}\partial_j\sigma}_{x_0,0}\leqslant C_3' t\sum_{\abs{\alpha}\leqslant 2} \norm{X^\alpha \sigma}_{x_0,1},\\
&\norm{(\omega_t(\tilde{e}_i)-\Omega_{ij}X^j)\sigma}_{x_0,0}\leqslant C_3''t^{\frac{1}{2}}\sum_{\abs{\alpha}\leqslant 1} \norm{X^\alpha \sigma}_{x_0,0},\\
&\norm{A_{s,t}(\tilde{e}_i)\sigma}_{x_0,0}\leqslant C_3''' tR\sum_{\abs{\alpha}\leqslant 1} \norm{X^\alpha \sigma}_{x_0,1},
\end{split}
\end{equation}
and therefore, there exists a $C_3>0$ such that
\begin{equation}
\norm{(\nabla_{s,t,\tilde{e}_i}-\nabla_{0,\partial_i}) \sigma}_{x_0,0}\leqslant C_3  \sum_{\abs{\alpha}\leqslant 2} t^{\frac{1}{2}}R\norm{X^\alpha \sigma}_{x_0,1}.
\end{equation}
it then follows that, there exists a $C>0$, such that, for any $\sigma, \sigma'\in  \in C_c^\infty(\mathbb{R}^n; S_{x_0}\otimes E_{x_0})$,
\begin{equation}
\inppro{(\Delta_{s,t}-\L)\sigma,\sigma'}_{x_0,0}\leqslant C t^{\frac{1}{2}}R\sum_{\abs{\alpha}\leqslant 4}\norm{X^\alpha \sigma}_{x_0,1} \norm{\sigma'}_{x_0,1},
\end{equation}
and the desired result then follows.
\end{proof}

Then, noticing that 
\[
(\lambda-\D_{s,t}^{x_0,2})^{-1}-(\lambda-\L_{s,0}^{x_0})^{-1}=(\lambda-\L_{s,0}^{x_0})^{-1}(\D_{s,t}^{x,2}-\L_{s,0}^{x_0})(\lambda-\D_{s,t}^{x_0,2})^{-1},
\]
the following estimate then hold, by \eqref{limrsv}.
\begin{cor}
Given any $m\in \mathbb{N}$, there exist constants $C_m>0$ and $M,N\in \mathbb{N}$, such that for all $\sigma\in C_c^\infty (\mathbb{R}^n, \Lambda^q (T_{x_0}M) \otimes E_{x_0})$,
\begin{equation}
\norm{\left[\left((\lambda-\D_{s,t}^{x_0,2})^{-1}-(\lambda-\L_{s,0}^{x_0})^{-1} \right)\sigma\right]^{(p+q)}}_{x_0,m}\leqslant C_m t^{\frac{1}{2}}R^{\frac{p}{2}+1}(1+\abs{\lambda})^N \sum_{\abs{\alpha}\leqslant M}\norm{X^{\alpha} \sigma}_{x_0, m}.
\end{equation}
\end{cor}

Now as in the previous sections, we establish the estimate as follows.
\begin{proof}[Proof of Proposition \ref{esterr}]
First the difference of the two operators in terms of the resolvents as
\begin{equation}
\begin{split}
&\exp{(-\D_{s,t}^{x_0,2})}-\exp{(-\L_{s,0}^{x_0})}\\
=&\frac{(-1)^{k-1}(k-1)!}{2\pi \sqrt{-1}}\int_\Gamma e^{-\lambda}\left((\lambda-\D_{s,t}^{x,2})^{-k}-(\lambda-\L_{s,0}^{x_0})^{-k}\right)\d\lambda.
\end{split}
\end{equation}
for any $k\in \mathbb{N}^\ast$. And, for the resolvents,
\begin{equation}
(\lambda-\D_{s,t}^{x,2})^{-k}-(\lambda-\L_{s,0}^{x_0})^{-k}
=\sum_{i=0}^{k-1}(\lambda-\L_{s,0}^{x_0})^{-i-1}(\D_{s,t}^{x,2}-\L_{s,0}^{x_0})(\lambda-\D_{s,0}^{x_0,2})^{-(k-i)}.
\end{equation}

In particular, take $Q_1,Q_2,\dots , Q_{m'} \in \{\nabla_{1,0,\partial_{i}}\}_{i=1}^n$, and define $\mathscr{Q}=Q_1,Q_2,\dots, Q_{m'}$, $\mathscr{Q}'=Q_{m'+1}, \dots , Q_{m}$ , it then follows that
\begin{equation}
\begin{split}
&\mathscr{Q}\left(\exp{(-u\D_{s,t}^{x_0,2})}-\exp{(-u\L_{s,0}^{x_0})}\right)\mathscr{Q}' \\
=&\frac{(-1)^{m-1}(m-1)!}{2\pi \sqrt{-1}u^{m-1}}\int_\Gamma e^{-u\lambda}\left(\sum_{i=0}^{m-1}\mathscr{Q}(\lambda-\L_{s,0}^{x_0})^{-i-1}(\D_{s,t}^{x,2}-\L_{s,0}^{x_0})(\lambda-\L_{s,0}^{x_0})^{-(m-i)}\mathscr{Q}' \right) \d\lambda.
\end{split}
\end{equation}
As in the previous section, there for all $m\geqslant 0$, $0\leqslant i \leqslant m-1$, there exist $C_m>0$, $N>0$ such that for any $\sigma \in C_\epsilon^\infty (\mathbb{R}^n, S_{x_0}\otimes E_{x_0})$, 
\begin{equation}
\norm{\mathscr{Q}(\lambda-\L_{s,0}^{x_0})^{-i-1}(\D_{s,t}^{x,2}-\L_{s,0}^{x_0})(\lambda-\L_{s,0}^{x_0})^{-(m-i)}\mathscr{Q}'  \sigma}_{x_0,0}\leqslant C_m t^{\frac{1}{2}} R^{\frac{n+1}{2}}(1+\abs{\lambda})^N \norm{\sigma}_{x_0, 0}.
\end{equation}
Then the desired estimate follows by precisely the same procedure as in the proof of Theorem \ref{knlest}.
\end{proof}

Now we show that the above estimate indeed guarantees Lemma \ref{limofker}.
\begin{proof}[Proof of Lemma \ref{limofker}]
Taking any $t\in (0,\frac{1}{2R}]$ and $X\in B_\epsilon(0)$, note that
\begin{equation}
\begin{split}
&t^{\frac{n+1}{2}}c_t(\hat{a})(\delta_t \K_s)(1;X)-\varepsilon(\hat{a})\K_{s,0}(1;X,0)\\
=&t^{\frac{1}{2}}c_t(\hat{a})\left(t^{\frac{n}{2}}(\delta_t \K_s)(1;X)-\K_{s,0}(1;X,0)\right)+(tc_t (\hat{a}) -\varepsilon(\hat{a}))\K_{s,0}(1;X,0).
\end{split}
\end{equation}
It follows from the explicit of $\K_{s,0}(1;X,0)$ that there exists $C_1>0$ such that for all $s\in [0,1]$
\begin{equation}
\sup_{X\in B_{\epsilon}(0)}\abs{\K_{s,0}(1;X,0)}\leqslant C_1 R^{\frac{n-1}{2}},
\end{equation}
and therefore, there exists $C_1'>0$ such that for all
\begin{equation}
\sup_{X\in B_{\epsilon}(0)}\abs{(tc_t (\hat{a}) -\varepsilon(\hat{a}))\K_{s,0}(1;X,0)}\leqslant C'_1 tR^{\frac{n-1}{2}}.
\end{equation}
On the other hand, write
\begin{equation}
\begin{split}
&t^{\frac{n}{2}}(\delta_t \K_s)(1;X)-\K_{s,0}(1;X,0)\\
=&t^{\frac{n}{2}}(\delta_t \K_s)(1;X)-\P^0_{x_0,s,t}(1;X,0)+ \P^0_{x_0,s,t}(1;X,0)-\K_{s,0}(1;X,0),
\end{split}
\end{equation}
where Proposition \ref{esterr} implies that there exists a $C_2>0$ such that
\begin{equation}
\sup_{X\in B_{\epsilon}(0)}\abs{\P^0_{x_0,s,t}(1;X,0)-\K_{s,0}(1;X,0)}\leqslant C_2t^{\frac{1}{2}}R^{\frac{n+1}{2}},
\end{equation}
while it follows from the same as argument Corollary \ref{estlocl} that there exist $C_3, \alpha>0$ such that
\begin{equation}
\sup_{B_\epsilon(0)}\abs{t^{\frac{n}{2}}(\delta_t \K_s)(1;X)-\P^0_{x_0,s,t}(1;X,0)}\leqslant C_3 e^{-\frac{\alpha \epsilon^2}{t}}.
\end{equation}
And finally, for all $X\in B_{\epsilon}(0)$,
\begin{equation}
\lim_{t\rightarrow 0}\abs{t^{\frac{n+1}{2}}c_t(\hat{a})(\delta_t \K_s)(1;X)-\varepsilon(\hat{a})\K_{s,0}(1;X,0)}=0.
\end{equation}
\end{proof}

\bibliographystyle{abbrv}
\bibliography{ASP}
\end{document}